\documentclass[reqno]{amsart}

\newcommand{\rev}[1]{#1}   

\usepackage[running]{lineno}

\usepackage{amsmath}
\usepackage{amsfonts}
\usepackage{amsopn}
\usepackage{amsthm}
\usepackage{amssymb}
\usepackage{graphics}
\usepackage{caption}
\usepackage{hyperref}
\usepackage{bm}
\usepackage{tabularx}
    \newcolumntype{L}{>{\raggedright\arraybackslash}X}

\newtheorem{theorem}{Theorem}[section]

\newtheorem{proposition}[theorem]{Proposition}
\theoremstyle{definition}

\theoremstyle{remark}

\numberwithin{equation}{section}
\newcommand{\COMMENT}[1]{}

\usepackage[dvips]{graphicx}

\newcommand{\R}{{{\rm I}\kern-.13em {\rm R}}}

\newcommand{\eps}{\varepsilon}

\renewcommand{\(}{\left(}
\renewcommand{\)}{\right)}

\newcommand{\Rb}{\bar{R}}
\renewcommand{\R}{R}
\newcommand{\D}{D_x}
\newcommand{\Dc}{D_{\theta}}

\newcommand{\xt}{\tilde{x} }

\newcommand{\thetat}{ {\tilde{\theta}} }
\renewcommand{\v}{v_1}
\newcommand{\w}{v_2}
\newcommand{\TT}{\mathbb{T}}

\begin{document}

\title[] {Modeling  the Interplay of Oscillatory Synchronization and Aggregation via Cell-Cell Adhesion}


\date{}
\author{Tilmann Glimm, Daniel Gruszka}
\let\thefootnote\relax\footnotetext{ {\rev{T. Glimm is the corresponding author (email glimmt@wwu.edu).}}}
\address{Department of Mathematics, Western Washington University, Bellingham WA 98225}

\bibliographystyle{siam}
\begin{abstract}
We present a  model of systems of cells with intracellular oscillators (`clocks'). 
This is motivated by examples from developmental biology and from the behavior of organisms 
on the threshold to multicellularity.
Cells undergo random motion and adhere to each other. The adhesion strength between neighbors depends on their clock phases in addition to a constant baseline strength.
The oscillators are linked via Kuramoto-type local interactions.
The model is an advection-diffusion partial differential equation with nonlocal advection terms. We demonstrate that synchronized states correspond to Dirac-delta measure solutions
of a weak version of the equation. 
To analyze the complex interplay of aggregation and synchronization,
we then perform a linear stability analysis of the incoherent, spatially uniform state. This lets us classify possibly emerging patterns depending on model parameters.
Combining these results with numerical simulations, we determine a range of possible far-from equilibrium patterns when baseline adhesion strength is zero: There is
aggregation into separate synchronized clusters with or without global synchrony; global synchronization without aggregation; or
unexpectedly  a ``phase wave" pattern characterized by spatial gradients of clock phases. A 2D Lattice-Gas Cellular Automaton model confirms and illustrates  these results.
\end{abstract}


\maketitle


\section{Introduction}

\subsection{Overview}
The spontaneous synchronization of coupled oscillators is a common occurrence in physics, biology and social sciences that has captured the interest of mathemtical modelers for decades \cite{winfree1987,pikovsky2003,boccaletti2018}. 
Well-studied, prominent phenomena are the emergence of synchrony in neural networks, i.e,. the synchronous firing of nerve cells \cite{singer1999,hansel1995,uhlhaas2009}.
and the synchronization of the flashing of fireflies (f.ex. \cite{faust2010,ramirez-avila2019}; see also the excellent article \cite{sokol2022}).   Synchronization of 
intracellular molecular oscillatory processes is also ubiquitous in developmental biology \cite{jiang2000,bhat2019,venzin2020, deneke2016}. 

The fundamental mathematical model of synchronization is the paradigmatic Kuramoto model \cite{kuramoto1984, strogatz1991,strogatz2000,witthaut2014} which has been studied extensively, in particular  on general complex networks \cite{rodrigues2016}. 
In general, synchronization of spatially distributed oscillators  on networks 
is the focus of very active research (see e.g. \cite{boccaletti2018,dorfler2014} and the references therein). Typically these networks, while often having complex topology, are static and do not encompass spatiotemporal rearrangements of the nodes  \cite{dorfler2014,rodrigues2016}. 

This is a very appropriate assumption for many applications,
e.g. neural networks, but it does not apply to aggregation
processes in which the individual oscillators are moving in space, and crucially, there is an interplay between clocks and motion.
The study of the feedback between motion and clock synchronization in systems of moving oscillators has only relatively recently come into focus.
Among those is the work of Tanaka and coworkers on systems of chemotactic oscillators \cite{tanaka2007,iwasa2010}.
The most influential work in this area is  O'Keeffe et al.'s seminal investigation of swarming oscillators \cite{okeeffe2017}, which has stimulated a number of further papers
\cite{sar2022,okeeffe2022} and has fascinating applications to e.g. the behavior of fleets of robots \cite{barcis2020}.
O'Keeffe et al. studied systems of Kuramoto-like oscillators (``swarmalators")  represented by mass points which attract or repel each other.

One important application of systems of interacting moving oscillators is aggregation of  biological cells mediated by cell-cell adhesion.
In such situations the swarmalator model of O'Keeffe et al. has some drawbacks, most prominently
that spatial attraction between two oscillators is independent of the distance between them. {\rev{(So it is a ``global attraction/local synchronization`` model.)}} Their model leads to highly symmetrical patterns of aggregation (a single disk or ring).
In contrast, adhesive interactions between cells are local, which, combined with random motion, typically leads to irregularly shaped aggregation clusters.

We therefore propose and investigate here a model where oscillators are coupled to adhesive forces, using models of cell-cell adhesion.
{\rev{This makes it a ``local adhesion/local synchronization" model.}}
Before describing our model in  more detail, we expand on our motivation: investigating
the aggregation of biological cells with internal biological clocks. 

\subsection{Aggregation and Molecular Clocks: Dictyostelia and Myxobacteria}
Two model organisms have been studied extensively: the slime mold Dictyostelia, in particular the best-studied species {\it{Dictyostelium discoideum}} \cite{bonner2008,oss1996,maree2001} and the slime 
bacteria Myxobacteria \cite{shimkets1982,zusmanetal2007,peruani2012,thutupalli2015}. While these two lineages are phylogenetically very distant (Dictyostelia are eukaryotes and Myxobacteria prokaryotes), they display convergent evolution as 
being on the threshold of multicellularity: Both can display solitary cell behavior (though only observed in a laboratory setting for Myxobacteria \cite{delangel2020}), where individual cells behave independently of each other. But in conditions of nutrient scarcity,
both transition into a developmental stage characterized by cell aggregation which culminates in the formation of multicellular structures called fruiting bodies.

The details of the mechanisms of aggregation are quite distinct for the two lineages and research of the two lineages has been almost completely independent.  In Dictyostelia, aggregation 
has been tied to a complex interplay of production of the chemotactic agent cAMP by cells and the movement of the same cells up the cAMP gradient. This is the basis of the
celebrated Keller-Segel model \cite{keller1970,horstmann2003}. In Myxobacteria, aggregation has been explained by coordinated cell motion mediated by alignment of the rod-shaped cells and synchronization of  
the reversal rate of neighboring cells \cite{wu2009,igoshin2004}. In a recent article, Del Angel et al. \cite{delangel2020}  pointed out that despite these differences,
there are some very profound similarities in the mechanisms of aggregation and directed migration between the two organisms. They argue that in both instances, it is a complex interplay between generic processes, i.e. mechanical or physicochemical processes that 
operate similarly in living and nonliving material and agent-like processes unique to living systems. The latter are mediated through networks of interaction with neighboring cells through mechanical and chemical signals. 
A paradigm for this is cell-cell adhesion interacting with the spontaneous emergence of synchronicity in intracellular oscillators. 
In Dictyostelia, the molecular basis is the periodic production of cAMP. (It is not clear if  these oscillations are an intrinsic 
property of individual cells or if they exist only as a population-level phenomenon \cite{gregor2010}.) In Myxobacteria, the periodic reversal
of directions is controlled by the  Frz and MglAB intracellular oscillators \cite{guzzo2018}.

\subsection{Model and Results}
In our bottom-up model of oscillators coupled to adhesive forces ('sticky clocks'), we use a PDE model
for a cell density $R(t,x,\theta)$ that depends on time $t$, position $x$ and clock value $\theta$. To model cell-cell adhesion, we adapt the model of Armstrong et al. \cite{armstrongetal2006}; see also \cite{dyson2013,buttenschonetal2018,chen2020,glimmetal2014,glimmzhang2020}, which was set up to  model cell sorting
by differential adhesion \cite{steinberg1970,glazier1993} in a PDE framework. In our approach,
adhesion is modeled by an advection-diffusion equation with nonlocal advection and Kuramoto-like interactions between oscillators via advection terms in the clock variable $\theta$.

There is a baseline adhesion strength given by a parameter $J_0$. An additional differential adhesion term depends on clock phases and is encoded by  $J$. The coupling strength of oscillators is denoted by $K$. 
For $K>0$, neighbors want to synchronize, for
$K<0$, they want to anti-synchronize. Positive $J$ means that cells' adhesion to neighbors is strongest if they are in the same clock phase; for $J<0$, the adhesion is strongest 
for cells in opposite phases.  (This choice of notation is motivated by O'Keeffe et al's  swarmalator model \cite{okeeffe2017}, where $J$ encodes attraction between particles.)

In this model, spatial aggregation and synchronization can co-occur and, depending on the parameters, can 
help or hinder each other. We perform a linear stability of the spatial homogeneous, incoherent (unsynchronized) state. This leads to 
a classification of the possible types of emerging patterns (``pattern onsets").  
The pattern onsets display either globally synchronized, locally synchronized or incoherent behavior with or without spatial aggregation. We then compare the linear pattern onsets with far-from equilibrium patterns from numerical solutions 
in the case of zero baseline adhesion ($J_0=0$). 
We find that the pattern onsets obtained by
linear stability analysis generally capture the far-from equilibrium behavior: Positive clock coupling $K$ leads to synchronization, where the value of the spatial adhesion strength $J$ controls whether there is spatial aggregation and if the synchronization is global or only local within aggregated clusters.
For negative $K$ and negative $J$, the incoherent state is stable and no aggregation occurs. We found a surprising result when $K$ is negative but small in absolute value and $J$ is sufficiently large.
In this case, small diffuse aggregates form. Oscillators are not synchronized, but instead {\rev{there}} are spatial gradients of clock phases, so that there is a correlation between clock phase and spatial position (a ``phase wave" pattern). 
These results are summarized in Figure~\ref{fig_conj} with examples in one spatial dimension in Figure~\ref{fig_example_comp}, and with a two-dimensional Lattice-Gas Cellular Automaton implementation in Figure~\ref{fig_lgca_results}.

The structure of the paper is as follows: In section~\ref{section_model}, we present the model, which consists of a diffusion-advection equations with nonlocal terms, and discuss its features. We formulate a weak version of the PDE in section~\ref{section_math}, which enables us to mathematical rigorously identify synchronized states as Dirac-delta functions. In section~\ref{section_linstab}, we perform a linear stability analysis, which
leads to a classification of pattern onsets. Finally, in section~\ref{section_numerics}, we investigate the far-from-equilibrium behavior of soultions by combining the results of the linear stability analysis with 
numerical solutions of the PDE (in one spatial dimension) and an implementation via a Lattice-Gas Cellular Automaton as described above.

\section{The Model}\label{section_model}
We consider the density $\R(t,x,\theta)$ of moving oscillators. Here $t$ is time, $x$ is spatial location and $\theta$ is the phase of the internal clock. The number $N$ of ocillators located in a 
region $U$ of $n-$dimensional space with phases between $\theta_1$ and $\theta_2$ at time $t$ is thus
\[
  N= \int_{U}\int_{\theta_1}^{\theta_2} R(t,x,\theta)\, d\theta\, d^nx.
\]
This setup is applicable e.g. to biological cells that have an internal molecular clock. 

{\rev{Our model is described by a conservation law  \cite{buttenschonetal2018}, i.e. an equation of the form 
$\partial_tR(t,x,\theta)=-\nabla\cdot J_{\text{space}}(x,t,\theta)-\partial_\theta J_{\text{clock}}(t,x,\theta)$. Here $\nabla$ denotes the gradient with respect to the spatial variable $x$. In this framework, different
biological phenomena can be modeled via different fluxes $J_{\text{space}}(t,x,\theta)$ in physical space and $J_{\text{clock}}(t,x,\theta)$ in clock space. Explicitly, our model equation is given by:}}
 \begin{linenomath}
\begin{align}
\frac{\partial R}{\partial t}=\D \nabla^2R- \nabla\cdot(\v(t,x,\theta;R)R)+\Dc \frac{\partial^2 R}{\partial\theta^2}-\,\frac{\partial }{\partial\theta}(\w(t,x,\theta;R)R). \label{diffeq}
\end{align}
\end{linenomath}
{\rev{The corresponding fluxes are $J_{\text{space}}(t,x,\theta)=-\D\nabla R + \v(t,x,\theta,R)R$ and $J_{\text{clock}}(t,x,\theta)=-\Dc\partial_{\theta}R+\w(t,x,\theta;R)R$. 
The first term in each flux is a diffusive flux in physical space and clock space with diffusion coefficients $\D$ and $\Dc$, respectively, assuming Fick's law. The term $\v(t,x,\theta,R)$ models the effective velocity of cells due to cell-cell adhesion. The term $\w(t,x,\theta,R)$ is the speed at which the clock phases advances.
The functional form for these terms are derived below.}}

The domain for (\ref{diffeq}) is
\[
      (x,\theta)\in [0,L]^2\times [0,2\pi]
\]
and we assume periodic boundary conditions, i.e. the density $R(t,x,\theta)$ is required to be $2\pi$-periodic in $\theta$
and $L-$periodic in $x$. Equivalently, we can think of $R(t,x,\theta)$ as function on the domain $[0,T)\times\Omega\times\TT$. Here $T>0$ and $\Omega={\mathbb{R}}^n/\sim$ 
is the quotient of  Cartesian space $\mathbb{R}^n$ by the equivalence relation that identifies the sides of a hypercube (square or cube in two or three dimensions) with each other, i.e. $(x,y)\sim(x+L,y)\sim (x,y+L)$. 
Similarly, $\TT=\mathbb{R}/2\pi\mathbb{Z}$ is a circle, namely the quotient $\mathbb{R}$ by the equivalence relation $\theta\sim\theta+2\pi$.

To set up the {\rev{of the clock velocity}} $\w(t,x,\theta,R)$, we assume that neighboring cells influence each others' cell clocks in a fashion similar to the Kuramoto model.
Specifically, we introduce a parameter $K$ below that describes and quantifies the strength of this influence. For $K>0$, neighboring cells seek to synchronize --
simply put, if my neighbor's clock is ahead of mine, my own clock speeds up a little. If $K<0$, this relation is reversed, and neighboring cells tend to {\em{anti-}}synchronize, i.e.
if my neighbor's clock is ahead of mine, my own clock tends to slow down.

Cells also move in the space. They undergo diffusion, and adhere to each other. {\rev{Adhesion is encoded by the effective spatial velocity $\v(t,x,\theta;R)$.}} The strength of the adhesion between neighbors depends on the absolute  difference $|\Delta\theta|$ of their clock phases. 
This is quantified by a second parameter $J$. For $J>0$, the adhesion strength is at a maximum at $|\Delta\theta|=0\, (\rm{mod } 2\pi)$  -- in other words ``like attracts like" and cells with the same phases adhere particularly strongly.
In the case $J<0$, the adhesion strength is at a {\em{minimum}} when $|\Delta\theta|=0$ and maximal when $|\Delta\theta|=\pi$, a situation best described by ``opposites attract." (See Table~\ref{tab1} for a summary of the meaning of the parameters $J$ and $K$. )

Specifically, these assumptions are modeled via the following functional terms for $\v$ and  $\w$:
 \begin{linenomath}
\begin{align}
\v(t,x,\theta;R)&=\int_{0}^{2\pi} \int_{D_{\rho}(0)}(J_0+J \cos(\theta-\thetat))\, \sigma(R(t,x+\xt,\thetat)\,) \frac{\xt}{|\xt|} d^n\xt\, d\thetat \label{v} \\ 
\w(t,x,\theta;R)&=\omega+K\, \int_{0}^{2\pi} \int_{D_{\rho}(0)}\sin(\thetat-\theta)\, R(t,x+\xt,\thetat)\, d^n\xt\, d\thetat \label{w}
\end{align}
 \end{linenomath}
Here $D_{\rho}(0)$ is a disk of radius $\rho>0$ centered at $0$ -- a template for the interaction neighborhood: Cells within a spatial distance of $\rho$ influence each others' phases and adhere to each other. For the function $ \sigma$ in the fluxes, one can typically choose a linear form or a saturating logistic
expression \cite{armstrongetal2006,glimmetal2014}:
 \begin{linenomath}
\begin{align}\label{sigma}
\sigma(R) = R \quad \text{or} \quad \sigma(R) = \frac{R_{\max}}{R_{\max} - 1} R \max\left(1 - \frac{R}{R_{\max}}, 0 \right) \text{ for some } R_{\max} > 0.
\end{align}
 \end{linenomath}
Here on the right, $R_{\max}$ is a maximum cell density. We will concentrate on the linear form $\sigma(R)=R$ in the analysis, but 
will use the logistic form for numerical simulations. The factor is chosen so that the two terms are equal for $R=1$. (The simulations
in section~\ref{section_numerics} use a constant value $R=1$ with small random noise for the initial conditions.)

The form of the adhesion velocity $\v$ follows Armstrong et al.'s model (\cite{armstrongetal2006}). 
The parameter $J_0$ is a ``baseline" adhesion strength. For $J=0$, the above equations are in fact identical to Armstrong et al's model. 
{\rev{Then the integral in (\ref{v}) is a a net direction obtained by averaging direction vectors over a neighborhood with the sensing radius $\rho$.
Directions are weighted by  cell density.
This encodes that cells are attracted by their neighbors and tend to move in their directions.}} Our addition is that the strength of attraction is modulated by the clock phase difference, described by the term $J\cos(\theta-\thetat)$, following a similar approach as 
in \cite{okeeffe2017}{\rev{, but with local attraction in a sensing radius $\rho$ as opposed to global attraction}}. The clock speed $\w(t,x,\theta,R)$ comes from the Kuramoto model \cite{kuramoto1984}. 

Note that by considering the temporal shift $\theta\to\theta-\omega t$, we can assume without loss of generality that $\omega=0$. 

There are also diffusive terms in $x$ and $\theta$. {\rev{The spatial diffusion coefficient is $\D$.}} Diffusion in {\rev{clock}} space
{\rev{encodes}} a small random component in the clock speed with diffusion coefficient $\Dc$. We will put particular emphasis on the case $\Dc=0$, which is the most natural assumption (no diffusion in clock-space). 
However, a small value of $\Dc$ is advantageous for numerical analysis because it prevents blowup of solutions. In any case, $\Dc$ should be ``small", in the sense that
the dimensionless quantity $\Dc\rho^2/\D$ satisfies
\[
\frac{\Dc \rho^2}{\D}<<1
\]
(See section \ref{phases_onset}.)
The equation (\ref{diffeq}) is complemented by an initial condition is
\[
 R(t=0,x,\theta)=R_0(x,\theta)
\]
where $R_0(x,\theta)$ is a nonnegative function.

{\rev{The model (\ref{diffeq}) is derived from assumptions about the net flux of the cell density, following the approach of the Armstrong cell-cell adhesion model \cite{armstrongetal2006}.
In \cite{buttenschonetal2018}, it was shown that the Armstrong model itself can be derived as the Master equation of a space-jump process, in which cells are explicitly modeled as individual jumpers in a discrete grid.
While it is outside the scope of this paper, we expect that a similar approach can be taken to derive our model as the Master equation of a space jump-process on a grid that includes
both spatial dimensions and a dimension that corresponds to the clock phase. The jump probabilities in clock space are governed by the Kuramoto model. In physical space, jump probabilities
would encode a bias towards other cells, but mediated with a clock-dependent term. In section~\ref{section_LGCA}, we set up a lattice-gas cellular automaton (LGCA) model which encodes the
same rules, alhough we do not prove formal equivalence to our PDE model  (\ref{diffeq}).}}

{\rev{The spatial modeling assumptions of our model (\ref{diffeq}) have similarities to the Vicsek model of swarming \cite{vicsek1995}. However, one key difference is
the interaction between particles: In the Vicsek model, particles tend to align their velocities, i.e. move in the directions of their neigbors {\it{motion.}}
In our model, and Armstrong's model of cell-cell adhesion, particles tend to move towards their neighbors, i.e. in the directions of  their neighbors {\it{location.}}
Correspondingly, in the Vicsek model, particles tend to swarm, i.e. travel together, whereas in the Armstrong model (and our model), cells tend to form stationary aggregates.}}

We note that while the model  (\ref{diffeq}) is defined for any values $J_0,J,K$, not all of them are meaningful from a modeling point of view.
The model hinges on the assumption that patterns are sufficiently coarse grained that a description using cell density 
is appropriate. Speficially, meaningful values of $J_0, J$ and $K$ are those that lead to patterns with a characteristic length scale not smaller than the cell interaction radius $\rho$, see 
Appendix~\ref{constr_j0jk} for further discussion.

\begin{table}
\begin{tabular}{ l | l | l}
Condition & Description & Shorthand\\
\hline
$J>0$ & max adhesion in-phase ($\Delta\theta=0$) & ``like attracts like"\\
$J<0$ & max adhesion out-of-phase ($\Delta\theta=\pm\pi$) & ``opposites attract"\\
$K>0$ & clock speeds up if behind neighbor's &   ``neighbors seek to synchronize"\\
$K<0$ & clock slows down if behind neighbor's &   ``neighbors seek to {\em{anti-}}synchronize"\\
\end{tabular}
\caption{Explanation of the parameters $J$ and $K$. Here $\Delta\theta$ is the difference of a cell's clock phase and its neighbor's phase.}\label{tab1}
\end{table}

\section{Nonnegativity, weak versions and steady state  solutions} \label{section_math}
It is useful to formulate weak versions of the equation (\ref{diffeq}). This enables us to state in a rigorous way that a certain family of Dirac-delta distributions are steady states.
First, let us emphasize that $R(t,x,\theta)$ is a density function with respect to spatial position $x$ and clock phase $\theta$. Not surprisingly, solutions remain nonnegative for nonnegative initial
conditions  $R_0(x,\theta)$. For positive clock diffusion $\Dc>0$, this follows from a maximum principle for {\rev{parabolic}} PDEs (\cite{evansbook}), but for completeness, we include a proof that covers the case $\Dc\geq 0$:

\begin{theorem} (Nonnegativity of solutions)
Let $R(t,x,\theta)\in C^2([0,T)\times\Omega\times\TT)$ be a strong solution of (\ref{diffeq}), where $T>0$ and $\Dc\geq 0$.
Suppose the initial condition $R(t=0,x,\theta)=R_0(x,\theta)$ is a nonnegative continuous function.
Then  $R(t,x,\theta)\geq 0$ for all $0\leq t<T$ and $x\in\Omega, \theta\in\TT$.
\end{theorem}
\begin{proof}
Write $w_i(t,x,\theta)=v_i(t,x,\theta;R(t,x,\theta))$ for $i=1,2$ For $0<T_1<T$, let 
\[
\lambda>\max_{0\leq t\leq T_1,x\in\Omega,t\in\TT}(|\nabla\cdot w_1|+|\nabla\cdot w_2|)
\]
Consider for $\eps>0$ the function
\[
R_\eps(t,x,\theta)=R(t,x,\theta)+\eps\cdot e^{\lambda t}.
\]
Then 
 \begin{linenomath}
\begin{align*}
\frac{\partial R_\eps}{\partial t}&=\D\nabla^2R_\eps+\Dc\frac{\partial^2R_\eps}{\partial\theta^2}-\nabla\cdot(w_1\, R_\eps)-\frac{\partial}{\partial\theta}(w_2\, R_\eps)+\eps e^{\lambda t}(\nabla\cdot w_1+\nabla\cdot w_2+\lambda)\\
                              &>\D\nabla^2R_\eps+\Dc\frac{\partial^2R_\eps}{\partial\theta^2}-\nabla\cdot(w_1\, R_\eps)-\frac{\partial}{\partial\theta}(w_2\, R_\eps)
\end{align*}
 \end{linenomath}
The function $R_\eps(t,x,\theta)>0$ cannot take on nonpositive values for $\eps>0$. Indeed, suppose that $R_\eps$ attains the value $0$ for the first time at time $t=t^*$ at a point $x=x^*$ and clock value $\theta=\theta^*$.
Then $\frac{\partial R_{\eps}} {\partial t}(t^*,x^*,\theta^*)\leq 0$. 
Also, we have that at those values $\nabla R=0$ and $\frac{\partial R}{\partial \theta}=0$, and, because it is a {\rev{minimum}}, $\nabla^2R_\eps\geq 0,$ and $\frac{\partial^2 R}{\partial \theta^2}\geq 0$.
But this gives 
\[
\frac{\partial R_\eps}{\partial t}(t^*,x^*,\theta*)\geq 0,
\]
by the above estimate, a contradiction.  The positivity of $R_\eps$ for all $\eps>0$ immediately implies that $R(t,x,\theta)\geq 0$ for all $t,x,\theta$.
\end{proof}

In keeping with the standard approach (\cite{perthamebook}), we can formulate a weak version of the problem: Say a function $R\in L^{\infty}([0,T)\times \Omega\times\TT)$ is a solution of (\ref{diffeq})
in the distribution sense if for all test functions $\varphi\in C^{\infty}_c ([0,T)\times\Omega\times\TT)$, we have
 \begin{linenomath}
\begin{align*}
  -\int_0^T\int_\Omega\int_0^{2\pi} R(t,x,\theta)& \left(\frac{\partial\varphi}{\partial t}+\v(t,x,\theta;R) \cdot\nabla\varphi+\D\nabla^2\varphi+\w(t,x,\theta;R)\frac{\partial\varphi}{\partial\theta}+\Dc\frac{\partial^2\varphi}{\partial\theta^2}\right)\,d\theta\, d^nx\,dt\\
 &=\int_{\Omega}\int_0^{2\pi}\varphi(t=0,x,\theta)\, R_0(x,\theta)\, d\theta\, d^nx,
\end{align*}
 \end{linenomath}
where $\v(t,x,\theta;R)$ and $\w(t,x,\theta;R)$ are given by (\ref{v}) and (\ref{w}), respectively.

For an even more general concept of a solution that permits for Dirac-delta-like distributions in the clock variable $\theta$, we adapt the approach of Hillen et al. \cite{hillenetal2010}. Let ${\mathcal{B}}(\TT)$ be the space of regular signed real-valued Borel measures on $\TT=\mathbb{R}/2\pi\mathbb{Z}$.
This is a Banach space if the absolute variation is taken as a norm, see \cite{hillenetal2010} for details.
Let ${\mathbb{X}}=L^{\infty}(\Omega,\mathcal{B}(\TT))$. So if $R\in\mathbb{X}$ and $x\in\Omega$, then $R(x)$ is a measure on $\TT$. Say $R\in\mathbb{X}$ is a weak steady state if the following equality holds
for all test functions $\varphi\in C^{\infty}(\Omega\times\TT)$:
 \begin{linenomath}
\begin{align}\label{defweakss}
0=\int_\Omega\int_{\TT} \left( \v(x,\theta,x;R)\cdot \nabla\varphi  +\w(x,\theta;R) \frac{\partial{\varphi}}{\partial\theta}+\D\nabla^2\varphi+\Dc\frac{\partial^2\varphi}{\partial\theta^2}   \right)\, R(x,d\theta) \, d^nx
\end{align}
 \end{linenomath}
Here $R(x,d\theta)$ denotes integration with respect to the measure $R(x)$ and
 \begin{linenomath}
\begin{align}
\v(x,\theta;R)&= \int_{D_{\rho}(0)} \int_{\TT}(J_0+J\,\cos(\theta-\thetat)\, R(x+\xt,d\thetat)\,\frac{\xt}{|\xt|} d^n\xt \\ 
\w(x,\theta;R)&=K\, \int_{D_{\rho}(0)}\int_\TT \sin(\thetat-\theta)\, R(x+\xt,d\thetat)\, d^n\xt
\end{align}
 \end{linenomath}
(These $\v,\w$ are Borel-measurable functions on $\Omega\times\TT$.)
With these definitions, we can now state that we have two classes of steady states.
\begin{theorem}
\begin{enumerate}
\item For any constant $\Rb>0$, the constant function $R\in C^{\infty}~(\Omega~\times~\TT)$
\[
R(x,\theta)=\Rb=const
\]
is a strong steady state solution. We refer to it as the {\em{incoherent state.}}

\item \label{dirac} Suppose $\Dc=0$. Let $\theta_0\in[0,2\pi)$ be any clock value. Suppose $R_0\in C^2(\Omega)$ is a strong solution of 
 \begin{linenomath}
\begin{align*}
\D \nabla^2R_0&- \nabla\cdot(w(x;R_0)R_0)=0\\ 
\quad & {\text{ with }} w(x;R_0)=\left(J_0+J\right)\int_{D_{\rho}(0)} R_0(x+\xt)\frac{\xt}{|\xt|} d^n\xt.
\end{align*}
 \end{linenomath}
Then the function $R\in\mathbb{X}$
\[
  R(x)=R_0(x)\, \delta(\theta-\theta_0)
\]
is a weak steady state. We refer to it as a {\em{(perfectly) synchronized solution}}.
\end{enumerate}
\end{theorem}
Here $\mu=\delta(\theta-\theta_0)$ denotes the point mass (Dirac) measure on $\TT$ concentrated at $\theta=\theta_0$, i.e. for any Borel set $B\subseteq\TT$, we have $\mu(B)=0$ if $\theta_0\notin B$
and  $\mu(B)=1$ if $\theta_0\in B.$
\begin{proof}
This is a straightforward computation. For (\ref{dirac}), note that
\[
\v(x,\theta;R)=\left(J_0+J\cos(\theta-\theta_0)\right) \int_{D_{\rho}(0)} R_0(x+\xt) \frac{\xt}{|\xt|} d^n\xt
\]
and 
\[
\w(x,\theta;R)=K \sin(\theta-\theta_0)\,\int_\Omega R_0(x)\, d^nx.
\]
Using these expressions, it is easy to verify that (\ref{defweakss}) holds.
\end{proof}

\section{Analysis of Phases} \label{section_linstab}
To analyze the behavior of equation (\ref{diffeq}) in dependence of the parameters $K$, $J$ and $J_0$, we perform a linear stability analysis of the incoherent steady state $R(x,\theta)=\Rb=const$.
The solutions of the linearized system give indications about the possible types of far-from-equilibrium patterns that can form. We use this to classify the emerging patterns (`pattern onsets'). 
We then compare these linearized patterns with numerical solutions of the original equation in Section~\ref{section_numerics}. 
\subsection{Linear Stability Analysis}\label{linstab}
For the linear stability analysis of the incoherent steady state $R(x,\theta)=\Rb=const$, we seek solution of the form 
\[
  R(t,x,\theta)=\Rb+c(t)e^{i(q\theta+k\cdot x)},
\]
where $q$ is the wave number of the clock phase $\theta$ and $k$ is the wave vector of the spatial coordinate $x$. The $2\pi-$periodicity of $R(t,\theta,x) $ in $\theta$ dictates that $q\in\mathbb{Z}$ takes on integer values. 
{\rev{Similarly, the periodic boundary conditions in space mean that only spatial wavevectors $k$ are considered for which $\frac{L}{2\pi}k$ takes on integer values.}} Solutions with temporally increasing coefficients $c(t)$ correspond to emerging patterns in space and in clock synchronization. For $|q|>0$, clocks synchronize; for nonzero $k$, spatial patterns form.

Utilizing this form in the model (\ref{diffeq}) and neglecting terms of order $c(t)^2$ and higher yields first-order equations for $c(t)$. The following Proposition summarizes the result of these computations in one and two spatial dimensions in dependence of the clock wave number $q$; see the appendix~\ref{appendix_lin_stab} for details on the derivation:
\begin{proposition}\label{linearization}
Let $d=1$ or $2$ denote the spatial dimension.
\begin{enumerate}
\item For $q=0$, the linearization is \label{q0}
\[
    \frac{d c}{dt}=\frac{\D}{\rho^2}\sigma_1(|\rho k|)\, c(t)  
\]
with $\sigma_1(|\rho k|)=-|\rho k|^2+2 \frac{J_0\Rb\rho^{d+1}}{\D}\,  F_d(|\rho k|)$; see below for the definition of the function $F_d$.
\item For $|q|\neq 0, 1$, the  linearization is \label{qn01}
\[
        \frac{d c}{dt}=(-|k|^2\D-q^2\Dc)\, c(t)
\]
\item For $|q|=1$, the  linearization is \label{q1}
\[
        \frac{d c}{dt}=\frac{\D}{\rho^2}\sigma_3(|\rho k|)\, c(t).
\]
with $\sigma_3(\rho k)=-|\rho k|^2 -\frac{\rho^2\Dc}{\D}+ \frac{J\Rb\rho^{d+1}}{\D}  F_d(|\rho k|) + \frac{\Rb K\rho^{d+2}}{\D} G_d(|\rho k|) $; see below for the definitions of $F_d$ and $G_d$.
\end{enumerate}
Here for each possible spatial dimension $d=1,2$, the pair $F_d(k^*)$ and $G_d(k^*)$ are two explicit functions defined via certain integrals. In one spatial dimension, they have the form  
\[
F_1(k^*)=2\pi(1-\cos(k^*)), \quad G_1(k^*)=2\pi\frac{\sin(k^*)}{k^*}.
\]
In two spatial dimensions, the functions $F_2$ and $G_2$ can be expressed explicitly via Bessel and Struve functions;
see the Appendix~\ref{appendix_lin_stab} for their explicit functional forms.
\end{proposition} 
The formula for $F_1$ is the same as in the clockless adhesion model of Armstrong et al. \cite{armstrongetal2006} and was given
there. Plots of the graphs of $F_2$ and $G_2$ are given in Figure~\ref{figFG}.

Note that the temporal eigenvalue in case \ref{qn01} is always negative. Therefore, in the linear stability analysis, only two cases of dependence on the clock phase $\theta$
need to be considered: Either uniform dependence (wavenumber $q=0$ in $\theta$-space) or waves with one mode ($q=\pm 1$). 

For the other two cases in Proposition~\ref{linearization}, there is a maximum temporal growth rate which corresponds to a certain mode in $\theta-$ and $x-$space. For $q=0$ (case \ref{q0}),
we denote the maximum growth rate by $\lambda_1$ 
For $q=\pm 1$ (case \ref{q1}), it is $\lambda_2$, the maximum (up to a constant factor) of $k^*\mapsto\sigma_3(k^*)$:
 \begin{linenomath}
\begin{align*}
  \lambda_1&=\max_{k^*> 0}\frac{\D}{\rho^2}\sigma_1(k^*)\\
  \lambda_2&=\max_{k^*\geq 0}\frac{\D}{\rho^2}\sigma_3(k^*)
\end{align*}
 \end{linenomath}
{\rev{Here the maxima are taken over the set of admissable  wavevectors $k^*=|\rho k|$, i.e. those that respect the periodic boundary conditions in space. Conservation of total cell mass forces the temporal growth rate of the mode $q=0, k^*=0$ to be zero, so we exclude $k^*=0$  from the definition of $\lambda_1$.}}
In addition, we denote by $k_2$ the value where the maximum of $\sigma_3$ is attained\footnote{We assume the smallest possible value of $k_2$ if the maximum is attained at multiple points.}:
\[
  k_2={\rm{argmax}}_{k^*\geq 0}\sigma_3(k^*)
\]

\begin{figure}
\includegraphics[width=0.9\textwidth]{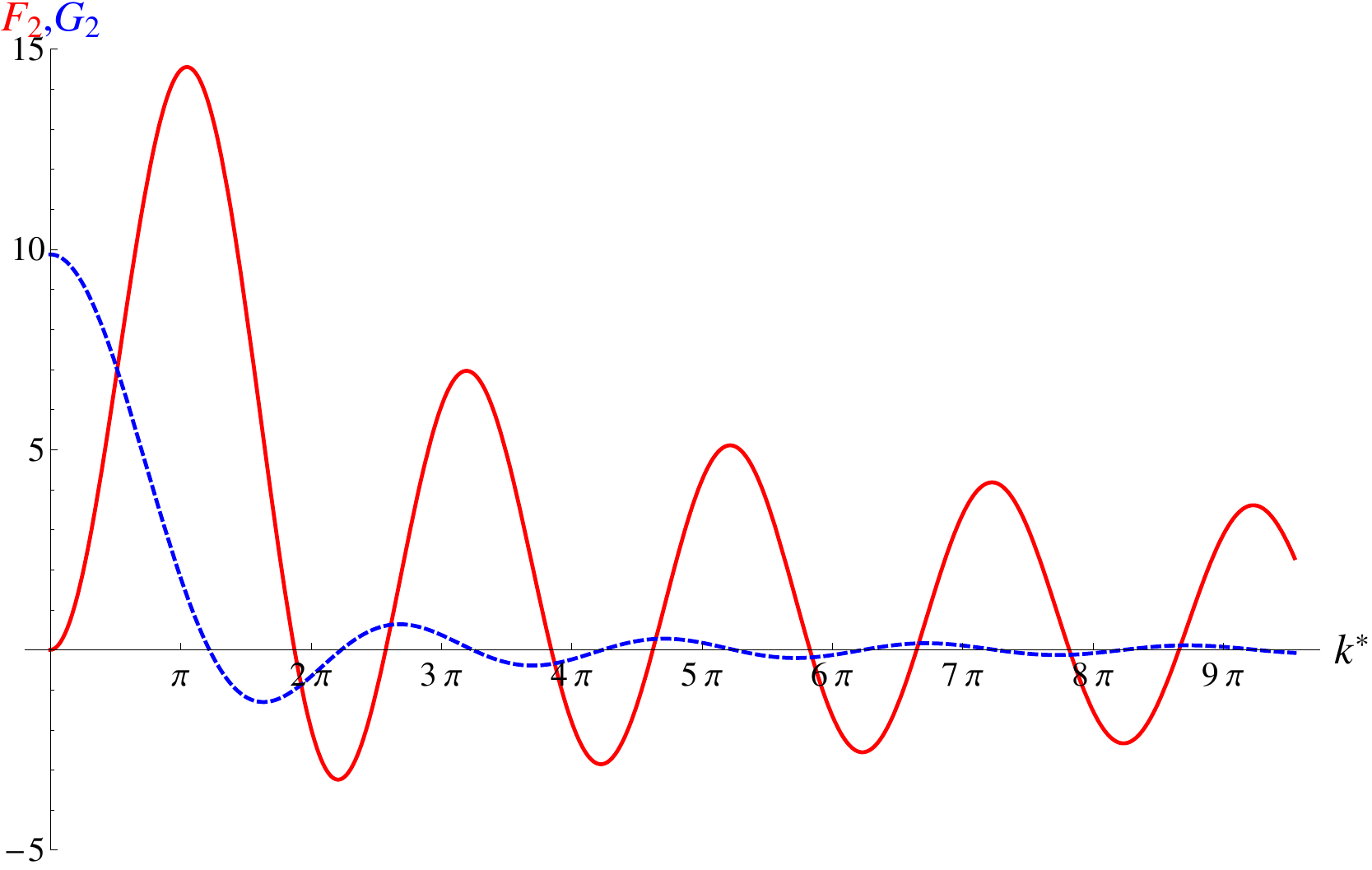}
\caption{Graphs of the functions $F_2(k^*)$ (red, solid) and $G_2(k^*)$ (blue, dashed); see Proposition~\ref{linearization}.}\label{figFG}
\end{figure}

\subsection{Classification of Pattern Onsets}
By the previous analysis, linearized solutions $R(t,\theta,x)$ are sums of the functions of the form $e^{\lambda t}e^{ikx}$ and  $e^{\lambda t}e^{\pm i\theta+ik\cdot x}$ with temporal growth rate $\lambda$ and spatial wavevector $k$.
This indicates the onset of certain types of possible patterns. Steady state patterns can be categorized {\rev{by  synchronization and spatial aggregation:
There may be no spatial aggregation ("uniform" pattern) or spatial aggregation along with no synchronization ("incoherent"), overall ("global") synchronization or local synchronization. }}
{\rev{Our terminology is summarized in Table~\ref{tab_terminology}. 
Accordingly, we classify the linearized solutions into six groups - uniform incoherent (UI),
uniform globally synchronized (UGS), unifom locally synchronized (ULS), aggregated incoherent (AI), aggregated locally synchronized (ALS) and aggregated globally synchronized (AGS).}}

{\rev{To carry out the classification}}, consider
linearized solutions of the form
 \begin{linenomath}
\begin{align*}
	R&(t,x,\theta)=\bar{R}+\Delta R(t,x,\theta)\text{, where}\\
           &\Delta R(t,x,\theta)=C_1e^{\lambda_1 t}\exp(ik_1\cdot x)+e^{\lambda_2 t}(C_{2,1}\exp(i\theta+ik_2\cdot x)+C_{2,2}\exp(-i\theta+ik_2\cdot x)).
\end{align*}
 \end{linenomath}
Here $\lambda_1$ and $\lambda_2$ are temporal growth rates of the corresponding modes. We assume that these modes are the fastest growing, i.e. that $\lambda_1$ is the maximum growth rate for the modes with $q=0$ 
(case \ref{q0} in Proposition~\ref{linearization}) with corresponding spatial wavevector $k_1\neq 0$, and $\lambda_2$ is the maximum growth rate for $q=\pm 1$ (case \ref{q1} in Proposition~\ref{linearization})  with corresponding $k_2$.
If the growth rate $\lambda_i$ is positive, the corresponding mode will grow in time; for negative $\lambda_i$, it will decay. This gives a classification of pattern onsets: Each combination of growth rates gives another category; for $\lambda_2>0$,
we also need to distinguish the cases $k_2=0$ [no spatial pattern, i.e. spatially uniform case] and $k_2\neq 0$. 
Considering only terms with positive temporal growth rate,  we then obtain the classification summarized in Table~\ref{tabonsets} and Figure~\ref{figonsets}. Here we also list for each case an example for the functional
form of the linearized pattern $\Delta R(x,\theta)$, mostly for the purpose of plotting an example for each pattern onset (Figure~\ref{figonsets}).

\begin{table}
\begin{tabularx}{\linewidth}{ |l | l | L| L|}
\hline
 Term & Abbr. & Description &  density $R(x,\theta)=\bar{R}+ \Delta R(x,\theta)$\\
\hline
\hline
(spatially) uniform & U & no spatial aggregation  & total spatial cell density $\int_0^{2\pi} R(x,\theta)d\theta$ is constant in $x$     \\
\hline
 aggregated & A& onset of spatial aggregation &  total spatial cell density $\int_0^{2\pi} R(x,\theta)d\theta$  varies with location $x$ \\
\hline
incoherent  & I& no synchronization & $R(t,x,\theta)$ is constant in $\theta$ for fixed $x$  \\
\hline
globally synchronized & GS &  onset of overall synchronicity &  $\Delta R(t,x,\theta)$ is a sine wave in $\theta$ with the same maximum  for all $x$\\
\hline
locally synchronized  & LS& onset of synchronicity within spatial clusters with distinct clock phases & $\Delta R(t,x,\theta)$ is a sine wave in $\theta$  for fixed $x$  with a maximum that varies by $x$\\
\hline
\end{tabularx}
\caption{{\rev{Terminology of pattern onsets. See Figure~\ref{figonsets} for graphical representations in one spatial dimension.
The same concepts are also illustrated with two-dimensional examples of a closely related discrete model in Figure~\ref{fig_lgca_results}.}} }\label{tab_terminology}
\end{table}

\begin{table}
\begin{tabular}{c | c |c ||c|c}  
$\lambda_1$ & $\lambda_2$ & $k_2$ & example $\Delta R(x,\theta)$ &  Name\\
\hline
$\lambda_1<0$  & $\lambda_2<0$ && $0$ &  {\bf{uniform incoherent (UI)}}\\
$\lambda_1>0$  & $\lambda_2<0$ && $-\cos(k_1x)$ &  {\bf{aggregated incoherent (AI)}}\\
$\lambda_1<0$  & $\lambda_2>0$ & $k_2=0$& $-\cos(\theta)$ &  {\bf{unif. globally synchronized (UGS)}}\\
$\lambda_1<0$  & $\lambda_2>0$ & $k_2\neq0$& $-\cos(\theta)\cos(k_2x)$ &  {\bf{unif. locally synchronized (ULS)}}\\
$\lambda_1>0$  & $\lambda_2>0$ & $k_2=0$& $-\cos(k_1 x)+\cos(\theta)$ &  {\bf{aggr. globally synchronized (AGS)}}\\
$\lambda_1>0$  & $\lambda_2>0$ & $k_2\neq0$& $-\cos(k_1x)-\cos(\theta)\cos(k_2x)$ &  {\bf{aggr. locally synchronized (ALS)}}\\
\end{tabular}
\caption{Classification of linearized solutions (pattern onsets); see Figure~\ref{figonsets}}\label{tabonsets}
\end{table}

\begin{figure}
\includegraphics[width=0.3\textwidth]{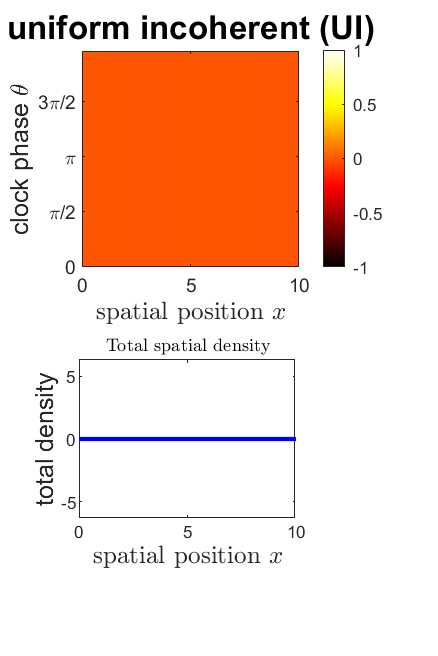}
\includegraphics[width=0.3\textwidth]{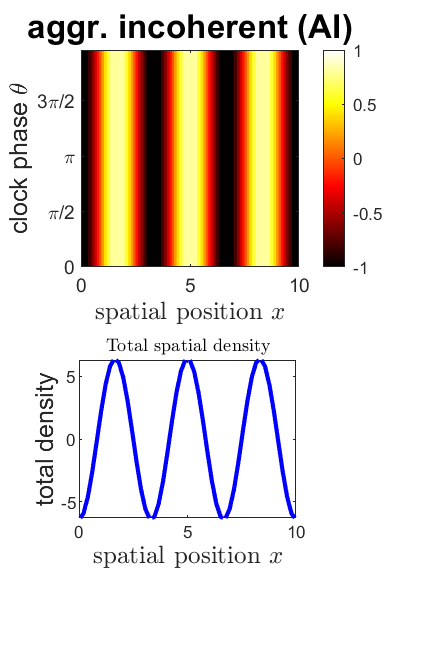}
\includegraphics[width=0.3\textwidth]{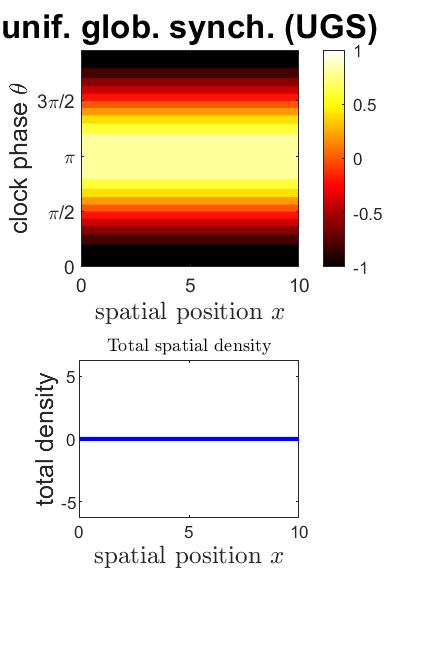}\\
\includegraphics[width=0.3\textwidth]{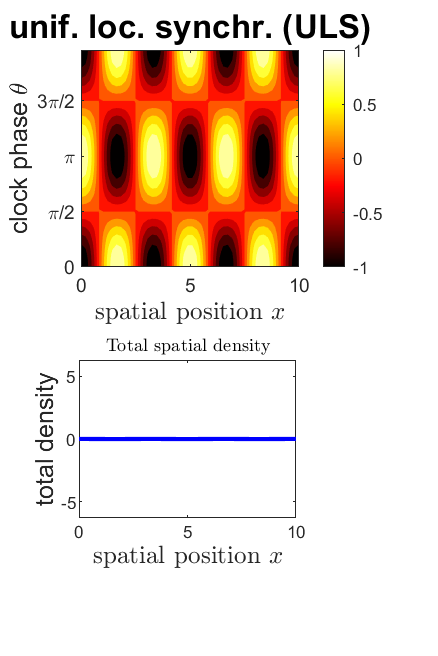}
\includegraphics[width=0.3\textwidth]{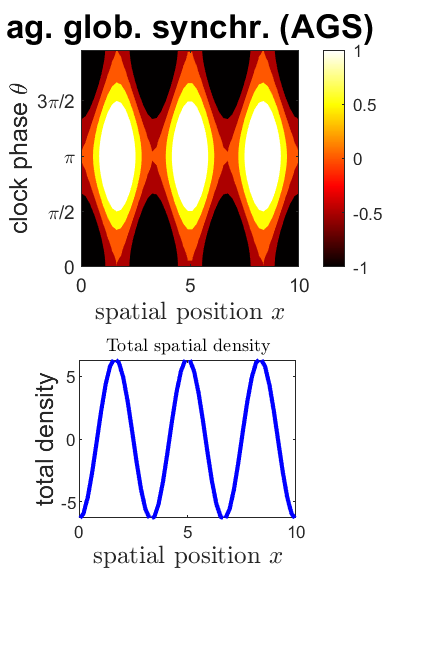}
\includegraphics[width=0.3\textwidth]{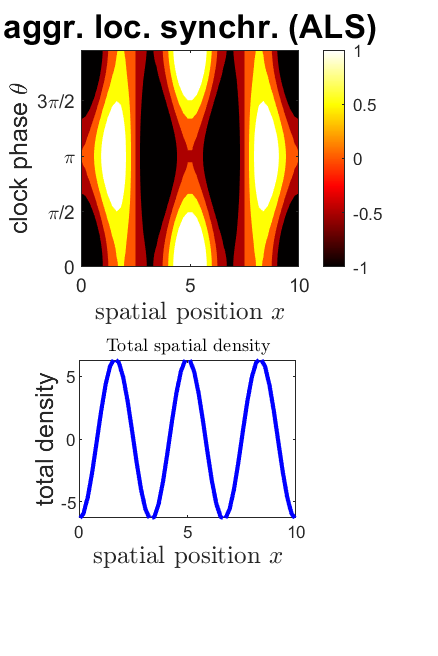}
\caption{Graphical representations of the six pattern onsets listed in Table~\ref{tabonsets} in one spatial dimension. For each pattern onset, the top panel shows a heat map of the linearized patterns $\Delta R(x,\theta)$ in $x-\theta-$ space, the bottom panel shows the corresponding
pattern in the total spatial density $\Delta R_{\rm{tot}}(x)=\int_0^{2\pi}\Delta R(x,\theta)\,d\theta$. For presentation purposes, we scaled densities to take values between $-1$ and $+1$. 
Here the length of the spatial interval is $L=10$. For AI, AGS and ALS, we chose $k_1=\frac{3\pi}{5}$.
Additionally for ULS,  $k_2=\frac{3\pi}{5}$ and for ALS,  $k_2=\frac{\pi}{5}$.}\label{figonsets}
\end{figure}

\bigskip

\subsection{Phase Diagram of Pattern Onsets} \label{phases_onset}
The linearization result in Proposition~\ref{linearization} along with the classification of possible pattern onsets (Table~\ref{tabonsets} and Figure~\ref{figonsets}) makes it possible to
determine a phase diagram of possible behavior of the system (\ref{diffeq}) according to the parameters $J_0, J$ and $K$. Define first the nondimensional parameters
\[
J^*=\frac{\bar{R}\rho^{d+1}}{\D}J, \quad J_0^*=\frac{\bar{R}\rho^{d+1}}{\D}J_0,\quad 
K^*=\frac{\bar{R}\rho^{d+2}}{\D}K
\] 
With this notation, the function $\sigma_1$ and $\sigma_3$ can be written as
 \begin{linenomath}
\begin{align*}
\sigma_1(k^*)&=-(k^*)^2+2 J_0^*\,  F_d(k^*)\\
\sigma_3(k^*)&=-(k^*)^2 -\Dc^*+ J^*  F_d(k^*) + K^* G_d(k^*) 
\end{align*}
 \end{linenomath}
where $\Dc^*=\rho^2\Dc/\D$. We restrict ourselves to parameter ranges that are consistent with the assumption 
that a meaningful cell density can be defined; see Appendix~\ref{constr_j0jk}. 

The classification of patterns in Table~\ref{tabonsets} is given in terms of the growth rates $\lambda_1$ and $\lambda_2$, which 
up to constants are the maxima of the functions $\sigma_1$ and $\sigma_3$. 
Thus which onset pattern occurs only depends on the  dimensionless values
$J_0^*, J^*, K^*$ and $\Dc^*$.
Note that we assume that  $\Dc^*\ll 1$, and in fact we are most interested in the case $\D^*=0$, i.e. the absence of diffusion in the
clock phase $\theta$. We concentrate on this case $\Dc^*=0$ in this section.

Pattern onsets depend on the signs of $\lambda_1$ and $\lambda_2$ and on whether $k_2$ is zero or nonzero (Table~\ref{tabonsets}).
The following two propositions  give us conditions on the parameters $J^*, K^*$ and $J^*_0$.
First, the sign of $\lambda_1$ is determined by $J_0$:
\begin{proposition}\label{propj0}
In each dimension $d=1,2$ there is a  value $J^d_{\rm{crit}}=\frac{1}{F''_d(0)}$ such that
\footnote{{\bf{We note the following sublety: It is possible that for small spatial domain sizes $L$, the maximum of $\sigma_1(k^*)$ for all real  $k^*>0$ may be positive, but still
the maximum on  those wavenumbers $k^*$ that respect the periodic boundary condition may be negative (so $\lambda_1<0$). 
To avoid such situations, the result is to be understood as valid for sufficently large domain sizes $L$. }}}
\[
\lambda_1<0 \text{ for }J_0^*<J^d_{\rm{crit}}\quad \text{ and } \lambda_1>0 \text{ for }J^*_0>J^d_{\rm{crit}}.
\]  
One computes
 \begin{linenomath}
\begin{align*}
            J^1_{\rm{crit}}&=\frac{1}{2\pi} \text{ in 1D}\\
	  J^2_{\rm{crit}}&=\frac{3}{2\pi^2} \text{ in 2D}
\end{align*}
 \end{linenomath}
\end{proposition}
\begin{proof}
The sign of $\lambda_1$ is determined by whether  $\sigma_1(k^*)$ has a positive global maximum for $k^*\geq 0$ or not.
By inspection, we have $\sigma_1(0)=\sigma_1'(0)=0$. It is thus the sign of $\sigma''(0)=-2+2J_0^* F''_d(0)$ which determines the sign of $\lambda_1$.
The above condition on $J^*_0$ follows immediately. 
\end{proof}
In terms of the model, Proposition~\ref{propj0} simply means that in every dimension, there is a critical value for $J^*_0$ above which spatial aggregation occurs spontaneously from 
incoherent, spatially homogeneous initial conditions.
\bigskip

The sign of $\lambda_2$ and whether the condition $|k_2|=0$ is satisfied is determined by $J^*$ and $K^*$.
The $K^*J^*$-parameter plane is divided into three regions:
\begin{proposition}\label{propjk}
In each dimension $d=1,2$ there is a functions $f_d(K^*)$ such that the 
following holds. (See Figure~\ref{figjk}.)

For any point $(K^*,J^*)$ in admissible parameter space (see Appendix~\ref{constr_j0jk}), we have\footnote{{\bf{As in Proposition~\ref{propj0}, the result is valid for sufficiently large spatial  domain sizes $L$; see the footnote there.}}}:
\begin{itemize}
\item If $K^*<0$ and $J^*<f_d(K^*)$, then $\lambda_2<0$.
\item  If $K^*>0$ and $J^*<f_d(K^*)$, then $\lambda_2>0$ and $k_2=0$.
\item  If $J^*>f_d(K^*)$, then $\lambda_2>0$ and $k_2>0$.
\end{itemize}
The function $f_d(k^*)$ is piecewise defined as follows:

For $K^*\geq 0$, 
\[
f_d(K^*)= \frac{2-K^*G_d''(0)} {F_d''(0)}=\begin{cases} \frac{1}{\pi}+\frac{1}{3}K^* & \text { for dimension } d=1\\
  \frac{3}{\pi^2}+\frac{3}{8}K^* & \text { for dimension } d=2
\end{cases}
\]
For $K^*<0$, the graph of $f_d(K^*)$ can be parameterized via 
\[
     \left(\begin{array}{c } K^* \\ f_d(K^*) \end{array}\right)=\frac{1}{G_d(\kappa)F_d'(\kappa)-F_d(\kappa)G'_d(\kappa)}
\left(\begin{array}{c  c}
\kappa^2 F'_d(\kappa)-2\kappa F_d(\kappa)\\
-\kappa^2 G'_d(\kappa)+2\kappa G_d(\kappa)
\end{array}\right)
\quad (\kappa> 0).
\]
\end{proposition}
\begin{proof}
Recall the definition that $\lambda_2$ is the (up to a constant) the maximum of $\sigma_3(k^*)=-(k^*)^2+J^*F_d(k^*)+K^*G_d(k^*)$ for $k^*\geq 0$ and $k_2$ is the smallest value at which it is attained.
Note that $\sigma_3(0)=K^*G_d(0)$ with $G_d(0)>0$.

Consider first the case $K^*>0$. Then $\lambda_2>0$. Moreover, by inspection of $F_d$ and $G_d$, for
$J^*\leq 0$, the maximum of $\sigma_3(k^*)$ occurs at $k^*=0$, so $k_2=0$ then. (See also Figure~\ref{figFG}.)
For small positive $J^*$, the maximum will still occur at $k^*=0$. As $J^*$ increases for fixed $K^*>0$, there is some critical value $J^*=f_d(K^*)$ above which the positive maximum occurs 
at a value $k^*>0$. The condition for this critical value is that $\sigma_3''(0)=0$. This means
\[
 \sigma_3''(0)=-2+J^*F''_d(0)+K^*G''_d(k^*)=0\quad\text{ at }J^*=f_d(K^*)
\]
This gives the formula for $f_d(K^*)$ for $K^*>0.$

Now consider the case  $K^*<0$. Then $\sigma_3(0)<0$. If $J^*<0$, then the maximum of $\sigma_3(k^*)$ is negative by our assumptions on the permissible range of $J^*$ and $K^*$ (see Appendix~\ref{constr_j0jk}),
so $\lambda_2<0$ then. For positive $J^*$, as $J^*$ increases, there is a critical value $J^*=f_d(K^*)$ above which a positive maximum occurs for some value $k^*=\kappa>0$. The condition for this is 
\[
\left\{\begin{array}{l}*\sigma_3(\kappa)=-\kappa^2+J^*F_d(\kappa)+K^* G_d(\kappa)=0,\\
 \sigma'_3(\kappa)=-2\kappa+J^*F'_d(\kappa)+K^* G'_d(\kappa)=0
\end{array}\right\} \quad\text{ at }J^*=f_d(K^*).
\]
Solving this for $(K^*,J^*)$ gives a paramaterization (with parameter $\kappa>0$) for the graph of $J^*=f_d(K^*)$ for $K^*<0$. 
\end{proof}
Propositions~\ref{propj0} and \ref{propjk} in conjunction with Table~\ref{tabonsets} now allow for a complete classification of the onset of patterns in dependence of the dimensionless parameters $J^*_0$, $J^*$ and $K^*.$ Figure~\ref{figcharonsets}
summarizes this graphically. 

\begin{figure}
\includegraphics[width=\textwidth]{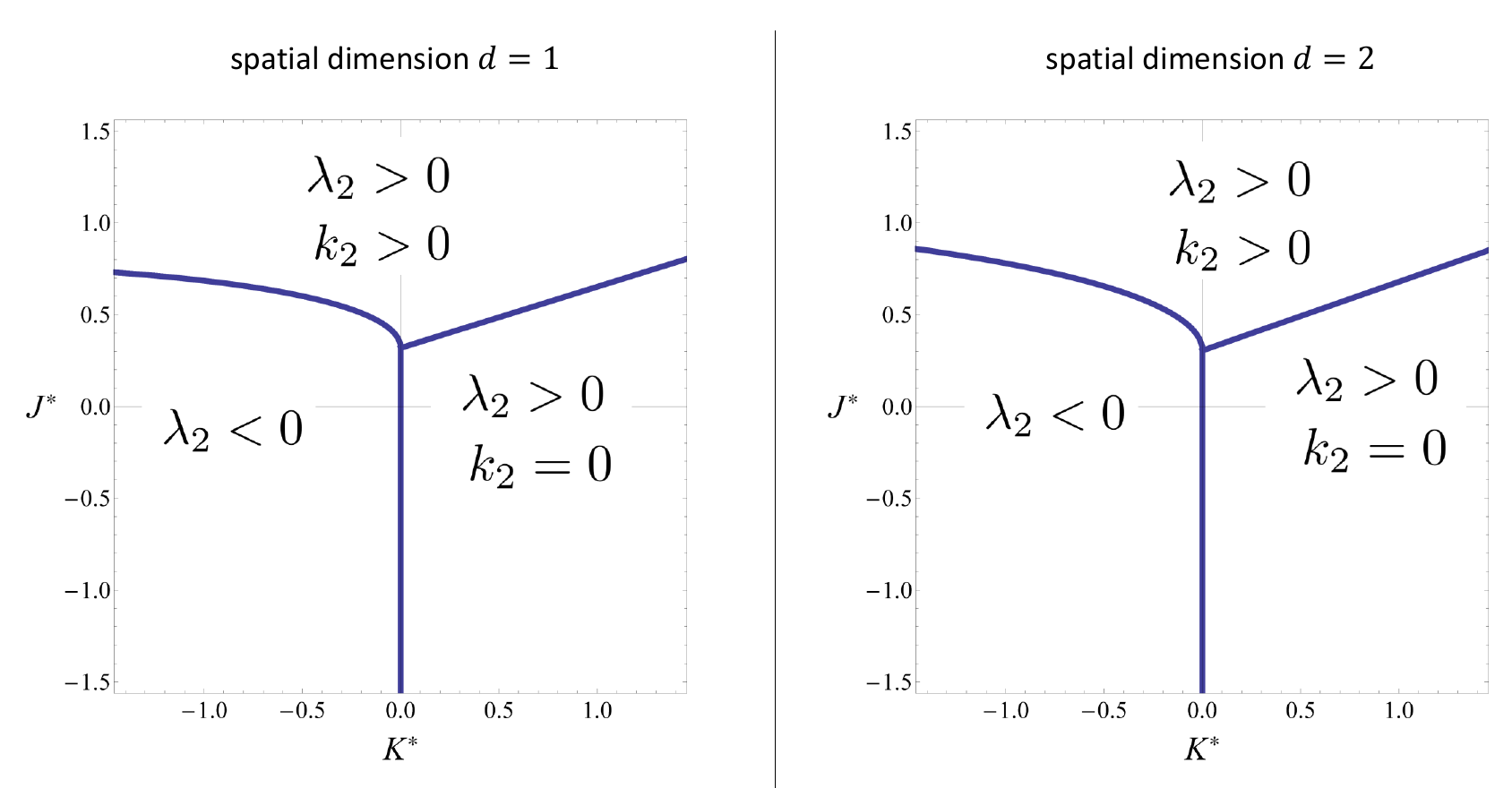}
\caption{Graphical representation of the result in Proposition~\ref{propjk}. Note that the images are slightly different in one and two spatial dimensions, although the overall appearance is very similar.}\label{figjk}
\end{figure}

\begin{figure}
\includegraphics[width=\textwidth]{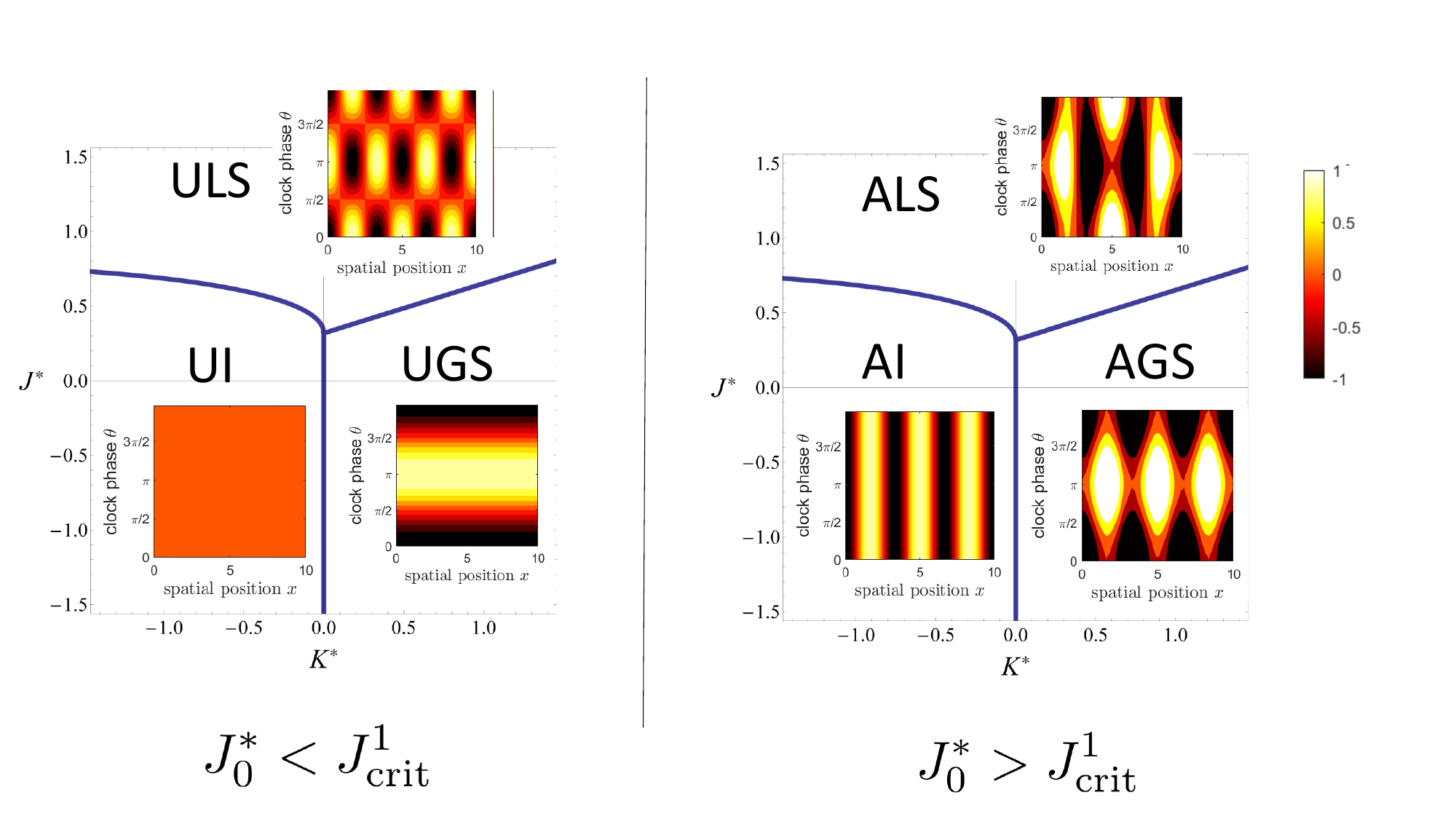}
\caption{Phase diagrams of onsets of patterning in $J^*K^*-$space. We show the case of $d=1$ spatial dimension; the case $d=2$ is very similar (Figure~\ref{figjk}).
See Table~\ref{tabonsets} for the classification of the different pattern onsets. Results are based on Propositions~\ref{propj0} and \ref{propjk}.
Left diagram: Case $J_0^*<J^1_{\rm{crit}}.$ In this case, there is no spatial aggregation and  pattern onsets are nonconstant in clock space only.
The three possibilities are uniform incoherent (UI), uniform locally synchronized (ULS) and uniform globally synchronized (UGS).
Right diagram: Case $J_0^*>J^1_{\rm{crit}}.$ In this case, spatial aggregation occurs simultaneously with patterning in clock space only.
The three possibilities are aggregated incoherent (AI), aggregated locally synchronized (ALS) and aggregated globally synchronized (AGS).
For the reader's convenience, we include the example densities patterns in $x\theta$-space from Figure~\ref{figonsets}. } \label{figcharonsets}
\end{figure}

\section{Far-From Equilibrium Behavior and Numerical Simulations}\label{section_numerics}
In the previous section, we used linear stability analysis to classify the {\it{onset}} of
patterns on the oscillator aggregation model from a spatially homogeneous, incoherent state (summarized in   Figure~\ref{figcharonsets}).
This does not necessarily give information about the {\it{far-from equilibrium}} behavior of the system.
There is one exception to this: the UI (uniform incoherence) case. Then all modes are linearly stable. As a consequence, no patterns can form in physical space or clock phase space. 
(For initial conditions that are homogeneous in the spatial coordinate $x$, i.e., $R(t=0,x,\theta)=R_0(\theta)$,
it is even possible to write out an explicit solution using Fourier analsyis and rigorously show that for $K<0$, the solution
converges to the incoherent state, see Appendix~\ref{appendixtheta}.)

In general, a natural expectation is that pattern onsets will generally lead to steady state patterns
with similar behavior. Specifically, if spatial  aggregation occurs in the pattern onset, i.e., if
the total spatial density $\bar{R}(t,x)=\int_0^{2\pi}R(t,x,\theta)d\theta$ has spatial peaks and valleys, this will lead to far-from equilibrium patterns 
with spatial aggregation.

Similarly, it is a natural expectation that pattern onsets of the ``globally synchronized" variety (UGS, AGS) actually  give
rise to globally synchronized far-from-equilibrium patterns. In other words, for zero clock diffusivity ($\Dc=0$), one expects convergence to a steady state pattern of the form
$R(x,\theta)=R_0(x)\delta(\theta-\theta_0)$ for some fixed $\theta_0$. This means that all 
oscillators are perfectly synchronized and the total spatial density is given by $R_0(x)$; see also Theorem~\ref{dirac}.
Note then that a linear stability
can be applied to $R_0(x)$. A formal argument gives the same result as Proposition~\ref{linearization} (\ref{q0})
with the only change being that we have to replace $J_0$ by $J_0+J$. By Proposition~\ref{propj0}, this means that now we 
expect patterns in the total spatial density for $J^*_0+J^*>J^d_{\rm{crit}}$.

These arguments yield that for the  pattern onsets UI and AGS, one expects similar behavior of the far-from equilibrium patterns of uniform (unaggregated) incoherent patterns and
globally synchronized patterns with spatial aggregations, respectively.
In the case of UGS (uniform globally synchronized) pattern onsets in $d$ dimensions, the above arguments suggest that for $J^*_0+J^*<J^d_{\rm{crit}}$, the far-from-equilibrium pattern
should be globally synchronized and also spatially uniform (i.e., without spatial aggregations). But for  $J^*_0+J^*>J^d_{\rm{crit}}$, one expects that even the UGS pattern onset
will give rise to aggregations once cells are sufficiently synchronized. Similar arguments apply to the ULS pattern onset.

In this section, we  compare the pattern onsets from  linear stability analysis 
with the far-from equilibrium patterns obtained by numerical simulations. We concentrate on the case $J_0=0$.
This means that cells with a $\frac{\pi}{2}$-phase shift neither repel nor attract each other;
in particular, cells in opposite phases repel each other and  there is no spatial aggregation in the incoherent phase.
We investigate the behavior in dependence of the parameter $J^*$ and $K^*$. We consider two computational models:
 The first is a numerical solution of the system (\ref{diffeq}) in one spatial dimensions. The second is an implementation of
a Lattice-Gas Cellular Automaton (LGCA) in two spatial dimensions representing  a discretized version of the PDE model (\ref{diffeq}).

\subsection{Numerical Solution of the PDE \ref{diffeq} in one spatial dimension}
We used the finite element solver FreeFEM++ \cite{freefem} to solve the PDE (\ref{diffeq}) in one spatial dimension for $J_0=0$. 
 Figure~\ref{fig_example_comp} summarizes the results with steady state patterns shown for different values of $J^*$ and $K^*$ obtained from
initial conditions with random small perturbations from the incoherent spatially uniform state $R_0(x,\theta)=1$.

To quantify the degree of synchronization of a distribution $R(t,x,\theta)$, we used  the $r$-parameter
first introduced by Kuramoto ($\sigma$ in Equation (5.4.13) in \cite{kuramoto1984}).  Here this is given by
\begin{equation}\label{r}
r=r(t)=\left|\int_{\Omega}\int_0^{2\pi}e^{i\theta}R(t,x,\theta)d\theta d^nx \middle/ \int_{\Omega}\int_0^{2\pi}R(t,x,\theta)d\theta d^nx\right|
\end{equation}
It represents the modulus of the mean value of $e^{i\theta}$. Note that $0\leq r\leq 1$. In the completely incoherent case $R(t,x,\theta)=R(t,x),$ we have
$r=0$. In contrast, perfect synchronization ($R(t,x,\theta)=R_c(t,x)\delta(\theta-\theta_0)$ for some $\theta_0\in[0,2\pi)$) corresponds to $r=1$.
%
%

The  phase diagram in Figure~\ref{fig_conj} summarizes the results using a combination of the linear stability analysis of pattern onsets (Figure~\ref{figcharonsets}, left panel)
and example numerical computations for certain pairs of parameters $(J^*,K^*)$ denoted by Roman numerals I--VII (Figure~\ref{fig_example_comp}).
The parameter space splits up into five regions separated by the curves $K^*=0$, $J^*=J^1_{\rm{crit}}=1/2\pi$ ($K^*>0$) and $J^*=f_1(K^*)$ (see Proposition~\ref{propjk}).
For the UI onset case, solutions also display spatial uniformity and clock incoherence (pairs I and II). The region of the UGS onset case displays two different possible behaviors:
In the region $J^*<J^1_{\rm{crit}}$, we obtain global synchronization without spatial aggregation (IV, V). In the slice of this region where  $J^*>J^1_{\rm{crit}}$, 
there is some spatial aggregation with a small amplitude in addition to global synchronization in the nonlinear system even though the linear pattern onset is spatially uniform (VI). 
The reason is that once synchronized, the adhesion strength  between cells becomes  sufficiently strong to  cause aggregation.

The ULS onset case is the most interesting. Our results indicate that for  $K^*>0$, the far-from-equilibrium behavior is as suggested by the linear stability analysis, but with additional subsequent spatial aggregation:
Local, but not global synchronization occurs, followed by spatial aggregation into segregated clusters. These clusters are synchronized within themselves, but not among each other (VII).
But for $K^*<0$, our results indicate a new kind of phase that is substantially different from all the linear patterns (III). This is the case of what in \cite{okeeffe2017} is referred to as a ``phase wave": Here spatial position is correlated with clock phase, i.e. the mean clock $\theta$ increases linearly with position $x$. We also observe some spatial aggregation along with the 
establishment of the phase wave pattern.

\begin{figure}
\includegraphics[width=0.7\textwidth]{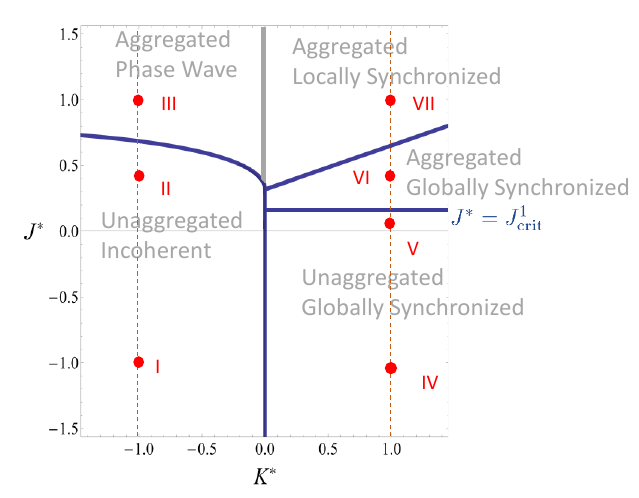}
\caption{Phase diagram of far-from-equilibrium patterning in one spatial dimension for Equation (\ref{diffeq}) in $J^*K^*-$space (with $J^*=0$). This is based on the onset diagram in Fig~\ref{figjk} (left panel).
The horizontal line on the $K^*>0$-half plane corresponds to the critical value $J^d_{\rm{crit}}$ in $d=1$ dimension.
For example numerical computation for the parameters indicated by Roman numerals, see
Figure~\ref{fig_example_comp}. Specifically, $(J^*,K^*)$ values are: I $(-1,-1)$, II $(0.4,-1)$, III $(1,-1)$, IV $(-1,1)$, V $(0.05,1)$, VI $(0.4,1)$, VII $(1,1)$. } \label{fig_conj}
\end{figure}

\begin{figure}
\includegraphics[width=\textwidth]{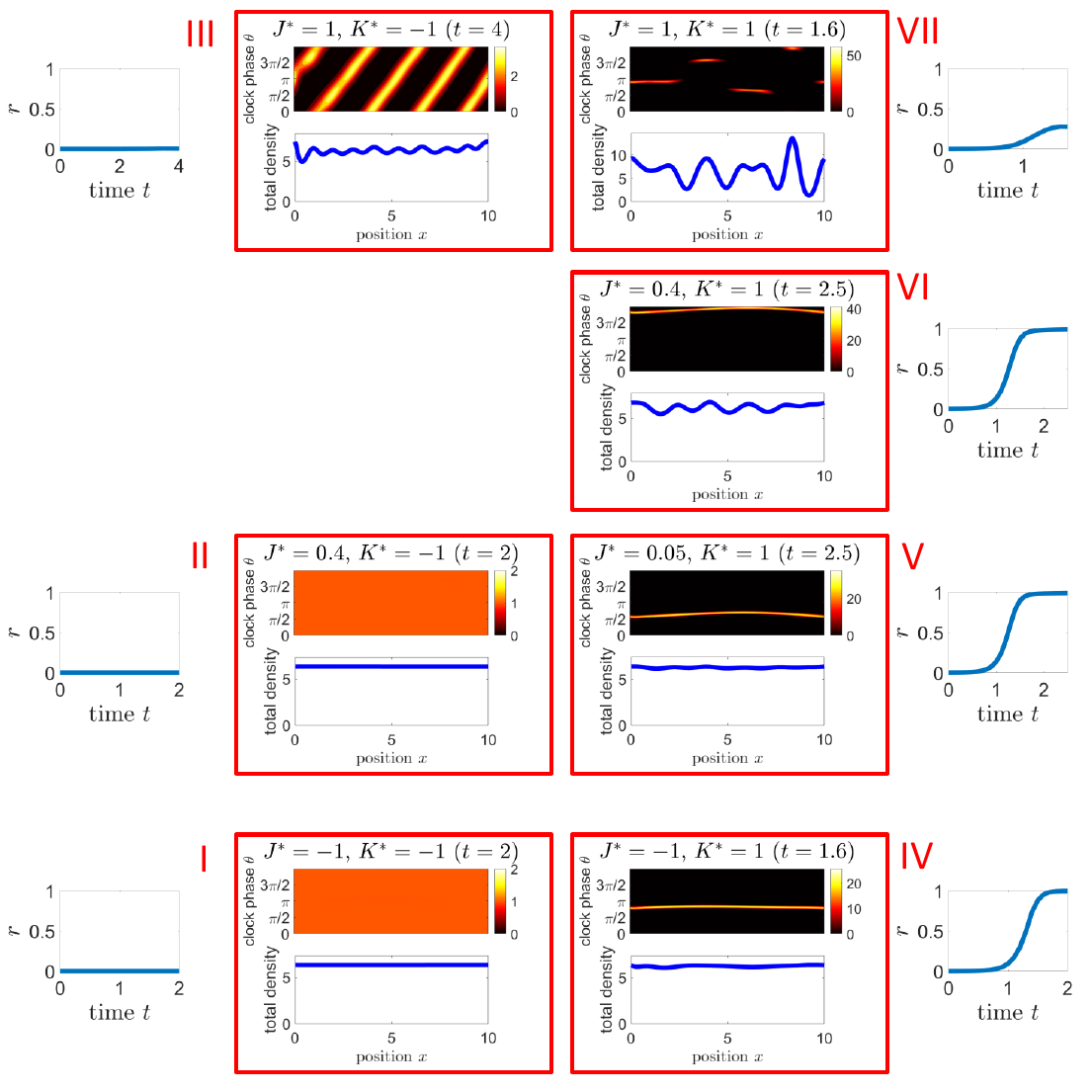}
\caption{Steady state distributions from numerical solutions of 
(\ref{diffeq}) with the parameters indicated in
Figure~\ref{fig_conj}. For each simulation, the graph of the Kuramato $r-$parameter as a function of time is shown. Parameters were $J_0=0$, 
$\D=1, \Dc=0.05, \rho=1$ on the spatial interval $[0,10]$ with periodic boundary conditions. 
A logistic function with $R_{{\max}}=30$ was chosen in the spatial flux term, see (\ref{sigma}).}\label{fig_example_comp}
\end{figure}

\newcommand{\s}{{\mathbf{s}}}
\newcommand{\clock}{{\bm{\theta}}}

\subsection{A Lattice-Gas Cellular Automaton model in two spatial dimensions}\label{section_LGCA}
{\rev{To investigate spatial patterns in two dimensions and to confirm the one-dimensional results, we set up a Lattice-Gas Cellular Automaton model based on the same principles as the PDE model (\ref{diffeq}). The fact that each cell is modeled as a point in a lattice provides a much better visualization
than the density-based approach of the PDE version (\ref{diffeq}); see Figure~\ref{fig_lgca_results}. We also found the PDE to be numerically challenging for two spatial dimensions as it is effectively a three dimensional problem with two spatial dimensions plus clock phase. In particular, the case $K>0$, when  neighboring clocks tend to synchronize, leads to sharp gradients along lines of constant $\theta$; see Figure~\ref{fig_example_comp}. Avoiding numerical instabilities requires very fine grids and small time steps.  
In contrast,  a cellular automaton scheme is more straightforward to implement numerically. Our results in this section also illustrate that our linear stability analysis and the analysis of the previous section are to some extent 
independent of the exact model implementation. Indeed, they also give insights into other models encoding the same basic rules of interaction between adhering oscillators.}} 

{\rev{Our}} Lattice-Gas Cellular Automaton model is based on models of cell-cell adhesion and aggrgeation (\cite{bussemaker1996} and \cite{deutschdormann2017} (chapter 7)); see  Figure~\ref{fig_lgca_schematics}.
The model is using a 2-dimensional $50\times 50$ hexagonal lattice with periodic boundary conditions. Each hexagonal node has six ``channels."  Each channel corresponds to a direction of motion. 
{\rev{Each of the six channels in a hexagon can either be occupied by one particle or unoccupied. So there are between zero and six particles in a hexagon at a given time, all in different channels. }} 
Additionally, each particle has a clock associated with it. The ``state" of a node is the collection of occupied channels in it with corresponding clock values, i.e., collections $\{s_i\}_{i=1,\ldots,6}$, where $s_i$ is the occupancy number (0 or 1) of the $i$th channel, and corresponding clock values $\{\theta_i\}_{i=1,\ldots,6}$.

The key idea of the model is that spatiotemporal evolution is modeled via two {\rev{alternating}} steps:{\rev{ an interaction step and a propagation step}. (See Figure~\ref{fig_lgca_schematics}.) In the propagation step, cells move from one node to a neighboring one in the direction given by their channel.
In the interaction step, the state of each node is updated. For this, each possible rearrangement of the particles in a node (state $\s=\{s_i\}_{i=1,\ldots,6}$ with clock values $\clock=\{\theta_i\}_{i=1,\ldots,6}$) is assigned an energy $H(\s,\clock)$. A new state $(\s,\clock)$ is then chosen with probability 
\[P(\s,\clock) = \frac{1}{Z}\exp(H(\s,\clock)),\]
where $Z$ is a constant chosen such that the sum of all probabilities is unity. In our case, the energy $H(\s,\clock)$ is calculated by considering how each particle  interacts with the neighboring particles. Specifically,

\begin{equation}\label{H}
H(\s,\clock) = \sum_{\substack{\text{channel }\\i=1,\ldots,6}} \sum_\text{neighbors}s_i\,\vec{c}_i \cdot \vec{p}_\text{neighbor}\,(J_0 + J\cos(\theta_\text{neighbor}-\theta_i))
\end{equation} 
Here $\vec{c}_i$ is the vector of direction corresponding to channel $i$ and $\vec{p}_\text{neighbor}$ is a vector in the direction of a (first order) neighbor's lattice site; see Figure~\ref{fig_lgca_h}. 
The clock values of the particle in channel $i$ and the neighbor are  $\theta_i$ and $\theta_\text{neighbor}$, respectively. 
The function $H(\s,\clock)$ encodes a flux along the direction of highest cell density, modulated by a term depending on the cell clock{\rev{, exactly  like  our PDE model (\ref{diffeq}) based on Armstrong et al's PDE model of cell-cell adhesion \cite{armstrongetal2006}; see also \cite{buttenschonetal2018}.}}

Finally, the clocks are updated based on a discrete Kuramoto-type model. {\rev{For this, every LGCA step is subdivided into $n$ subintervals of length $\Delta t=1/n$. (In simulations, we used $n=10$.)}  In each subinterval, the change in clock value of the particle at channel $i$ is given by
\[ \Delta\theta_i =(\omega + K\frac{1}{N_i}\sum_{j=1}^{N_i} \sin(\theta_j-\theta_i))\Delta t\] 
Here the sum is over the first order neighbors of the particle in channel $i$ and the particles in the other channels of the same hexagonal node (total $N_i$ particles). As in the PDE model (\ref{diffeq}), we can assume without loss of generality that $\omega=0$, by shifting the frame in clock phase by
the term $\omega t$, where $t$ is the number of time steps.

As in the PDE case (Equation (\ref{r})), one can define an $r-$parameter which quantifies the degree of {\em{local}} synchronization. In the LGCA model, this is given by
\[
r=\left|\frac{1}{N}\sum_{k=1}^Ne^{i\theta_k}\right|,
\]
where $N$ is the total number of particles. Again $r=1$ means complete synchronization and $r=0$ complete incoherence. We also considered $r_{\rm{local}}$, a local variation of $r$, where the sum is taken at each node over the first-order neighbors only and then averaged over all nodes. This quantifies
the degree of local synchronization.

Note that the model contains parameters $J_0,J$ and $K$ whose meaning is completely analogous to that of the PDE model (\ref{diffeq}), and equally the interpretations in Table~\ref{tab1} apply here.

One thus expects similar behavior as the displayed by PDE  version (\ref{diffeq}) displayed in Figure~\ref{fig_conj}. (Note that the phase diagram is for one spatial dimension,
but as shown in Figure~\ref{figjk}, the 2D case is substantially the same, at least its linearization.)
To test this, we conducted numerical simulations of this LGCA model with $J_0=0$ and various values of $J$ and $K$. As illustrated in 
Figure~\ref{fig_lgca_results}, this two dimensional LGCA model shows the same phases as the PDE model with the possible exception of the ``aggregated global synchronized" phase.
While the results of Figure~\ref{fig_conj} do not apply quantitatively to the LGCA model, their qualitative behavior is equivalent: For negative $K$ and $J$, the LGCA model also displays incoherent clock phases and lack of aggregation.
For positive $K$ and negative $J$, global synchronization occurs without spatial aggregation. For positive $J$ and $K$, we observed local aggregation into separate clusters, with synchronization within, but not among clusters. {\rev{For large $K$ and small $J$, aggregation occurs along global synchronization. In the LGCA, we observed for $J=0.1, K=10$ that aggregation into diffuse, locally synchronized clusters happened first, with subsequent synchronization of the clusters.}}
Finally, for negative values of $K$ close to 0 and large values of $J$, we observed a phase apparently equivalent to the ``aggregated phase wave" in Figure~\ref{fig_conj}. Here, some spatial aggregation occured, but the individual clusters did not synchronize. 
They were not completely incoherent either though, as indicated by its local $r-$value $r_{\rm{local}}$, which was much larger than for the incoherent phase; see Figure~\ref{fig_lgca_results}. This indicates a partial sorting by phases within clusters.
\begin{figure}
\includegraphics[width=\textwidth]{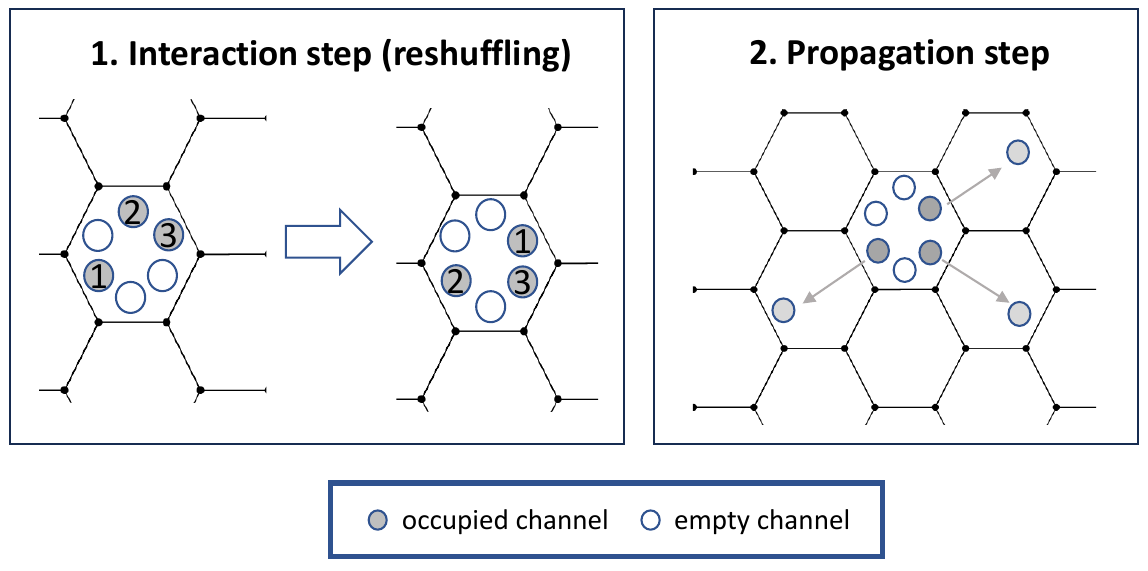}
\caption{Schematics of the Lattice-Gas Cellular Automaton model \cite{deutschdormann2017,bussemaker1996} : {\rev{Spatiotemporal evolution is modeled via two alternating steps on a hexagonal lattice. Each lattice site has six channels, which each can be occupied by a single particle (a cell) or empty.
In the interaction steps (left), particles within each hexagon are shuffled among the channels according to probabilistic rules. (Each particle has its own clock, so they are distinguishable. This is indicated in the example by labels 1 through 3.) In the propagation step (right), particles move to a neighboring hexagon according to their channel.}}   }\label{fig_lgca_schematics}
\end{figure}

\begin{figure}
\includegraphics[width=0.4\textwidth]{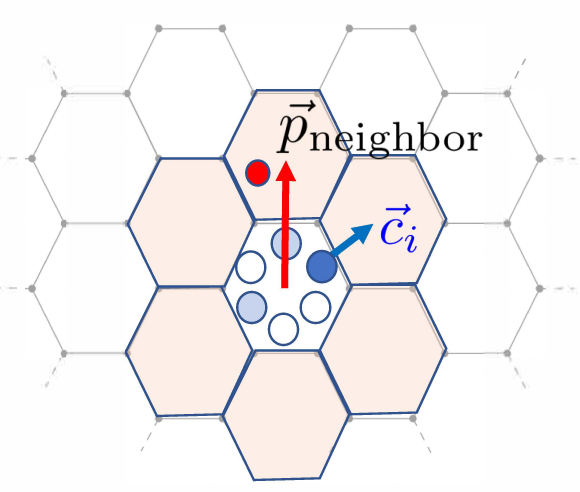}
\caption{Sketch for definition of the energy $H$ in Equation (\ref{H}): {\rev{In the interaction step, reshuffling of particles among channels is biased towards rearrangements with higher energy $H$. 
The sketch shows a possible rearrangement with a particle (blue)  with velocity vector $\vec{c}_i$ and one of its neighbors (red) with position vector $\vec{p}_\text{neighbor}$. The pair contributes the term $\vec{c}_i \cdot \vec{p}_\text{neighbor}\,  \phi(\theta_\text{neighbor},\theta_i)$  to $H$,   
where $\phi(\theta_\text{neighbor},\theta_i)=J_0 + J\cos(\theta_\text{neighbor}-\theta_i)$.   So e.g. if $\phi(\theta_\text{neighbor},\theta_i)>0$, there is a resulting bias towards rearrangements where the velocity vector $\vec{c}_i $ is directed towards the neighbor.}}  }\label{fig_lgca_h}    
\end{figure}

\begin{figure}
\includegraphics[width=\textwidth]{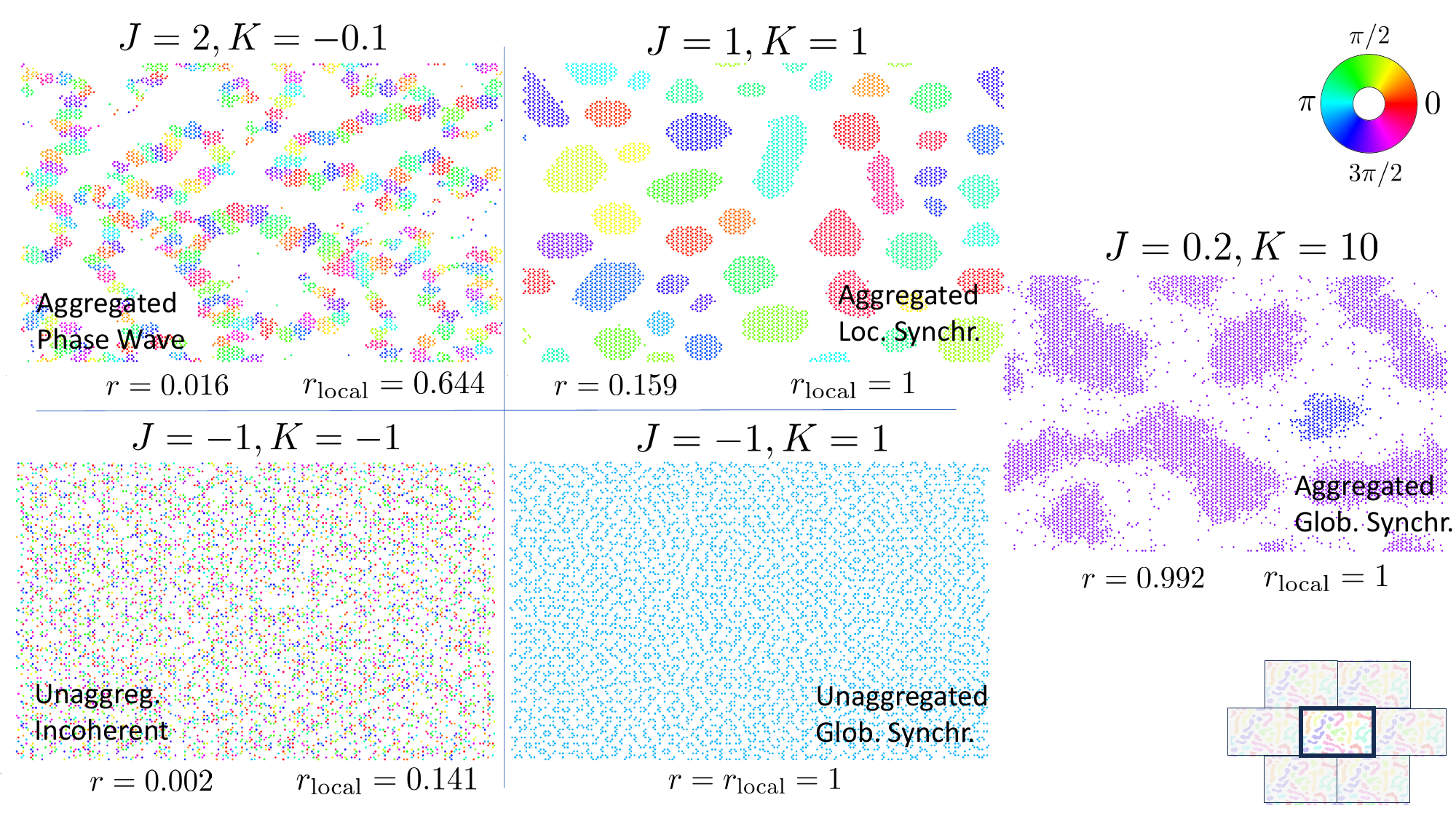}
\caption{Typical results for $J_0=0$ and different values of $J$ and $K$ in the LGCA model. Note that qualitatively, these phases correspond to
the ones obtained with the PDE model for Figure~\ref{fig_conj}. Different phases are indicated by colors; see the color wheel for reference.
The images show configurations obtained from random initial conditions after 2000 iterations on a 50$\times$50 hexagonal grid with periodic boundary conditions. 
{\rev{(The boundary conditions are illustrated in the bottom right corner.)}}
There were $N=6,000$ particles in each simulation, corresponding to a confluency (percentage of occupied channels overall) of $40\%$.}
\label{fig_lgca_results}
\end{figure}

\section{Discussion}
We have investigated the interplay of oscillators (cells with intracellular 'clocks') that undergo random motion and adhere to each other.  The strength of adhesion depends on 
the relative clock phases and neighbors can sync or anti-sync. A linear stability analysis allows for  classification of pattern onsets. We further investigated far-from equilibrium patterns for
the case that cells with a $\frac{\pi}{2}$-phase shift neither repel nor attract each other ($J_0=0$). We identified five different types of patterns: a completely synchronized state without spatial aggregation; 
aggregation into clusters that are internally synchronized, but out of sync with each other; an incoherent, unaggregated state; and a phase wave state with diffuse clusters and spatial gradients of clock phases.
A fifth state shows aggregation into diffuse, globally synchronized clusters. 

{\rev{Some of these phases of our local adhesion/local synchronization model show similarities to the global attraction/local synchronization swarmalator model of O'Keeffe et al. \cite{okeeffe2017}. Specifically,
their ``splintered wave" phase is analogous to the aggregated local synchronized phase (Figure~\ref{fig_lgca_results}) of our model in showing locally synchronized clusters. In the swarmalator case, this happens when neighboring cells are weakly coupled and {\it{de}}synchronize (small negative $K$),
but cells of the same phase attract each other strongly (large positive $J$). When $K$ becomes positive, these local clusters collapse into one synchronized disk-shaped aggregate. In our model, locally synchronized aggregates persist for $K>0$ and $J>0$ because adhesion is local (i.e. only between neighboring cells).
Another phase with paralles between the two models is O'Keeffe's ``active phase wave", where cells move in opposite directions in two concentric rings and cycle through the clock phases.
This leads to an apparent steady state distribution of clock phases along these rings, i.e. a gradient of clock phases. This is somewhat similar to the '``aggregated phase wave'' state in our model, where also spatial gradients
of clock phases exist, best seen in the one-dimensional results in Figure~\ref{fig_example_comp}.}}

A number of open issues present themselves: First, we investigated the far-from-equilibrium behavior only in the natural case of zero baseline adhesion $(J_0=0)$. Our pattern onset classification
extends to positive and negative $J_0$ as well. In particular, the case $J_0>J^d_{\rm{crit}}$ (so that cells aggregate even in the incoherent state) 
promises other types of behavior and an in-depth investigation could be fruitful. We also did not investigate the spatial dynamics of synchronization in our model in detail, i.e., how exactly synchronization patterns arise in time and space,
which again is an interesting avenue for further work.

There are also interesting analytical questions: Numerical simulations indicate that for postive $K$ and negative $J$, solution converge to a Dirac-delta point measure in clock-space and uniform distribution in space when the clock diffusivity is zero ($\Dc=0$).
Can this be proved rigorously? Even more general, do solutions always converge to function of the form $R(x)\delta(\theta-\theta_0)$ for $K>0$? 
For these question, a further examination of the properties of solutions of the  measure-theoretic formulation sketched in section~\ref{section_math} would be helpful (e.g. proving existence and uniqueness). 

As discussed in the introduction, one of our much broader future goal is to use this model and others of its type to  investigate generic mechanisms of aggregation, synchronization and  pattern formation in cell populations. This pertains to organisms on the threshold to multicellularity like Myxobacteria \cite{igoshin2004},
and Dictyostelia \cite{maree2001}, and also to developmental processes like pattern formation mediated by oscillatory signaling in the Notch-Delta circuit \cite{kageyama2018}).
Our hope is to explore broad general principles in the spirit of the conceptual framework of
Dynamical Pattern Modules (DPMs) \cite{newman2009}. Such DPMs are units of gene products associated with physical processes they mobilize, forming  a proposed ``pattern language" for pattern formation inm developmental biology.

Finally, there are some interesting sociological interpretations of our model in the context of polarization of political opinions and the formation of echo chambers
aided by social media \cite{cinelli2021}. (Mathematical models of these phenomena typically use network models, e.g. \cite{baumann2020}.)
 Our model's behavior in the aggregated, locally synchronized state  
(clusters that are synchronized within themselves, but not between each other) shows clear formal similarities to such echo chambers. Our results indicate that this is a natural outcome if neighbors tend to synchronize and ``like attracts like"  ($J>0$). One way 
to achieve global synchronization without formation of such echo chambers is then to  consider negative $J$; the case of  ``opposites attract." This suggest that a way out of the societal polarization is to
get people with different views in contact -- a truism maybe, but this certainly suggests that the model has the potential for less trivial sociological investigations.

\subsection*{Acknowledgements}
The authors acknowledge the financial support of the John Templeton Foundation (\#62220). The opinions expressed in this paper are those of the authors and not those of the John Templeton Foundation.

\bibliography{sync_adhesion_glimm_gruszka}

\newpage
\appendix
\section{Linear Stability Analysis}\label{appendix_lin_stab}
We summarize the computation for the linear stability analysis here whose results are stated in Proposition~\ref{linearization}. 
We only cover the case of $d=2$ spatial dimensions. The simpler case of one spatial dimension ($d=1$) is a straightforward
modification and we leave out the details that yield the corresponding functional forms of $F_1$ and $G_1$.

Consider solutions of the form
\begin{equation} \label{linR}
  R(t,x,\theta)=\Rb+c(t)e^{i(q\theta+k\cdot x)},
\end{equation}
to (\ref{diffeq}). Here $\Rb$ is a constant solution, i.e. an incoherent solution where all clock phases have equal spatially constant density.
The vector $k\in{\mathbb{R}}^2$ is the spatial wavevector. Periodic boundary conditions in the clock phase $\theta$ force that $q\in\mathbb{Z}$.
Below, we compute the terms $\nabla\cdot(\v(t,\theta,x;R)R)$ and $\frac{\partial}{\partial\theta}(\w(t,\theta,x;R)R)$ explicitly:

\bigskip
\noindent{\underline{\bf{Term } $\mathbf{\nabla\cdot(\v(t,\theta,x;R)R)}$}}

\medskip
One computes that
\begin{linenomath} 
\begin{align}
\v(t,x,\theta;R)&=\int_{0}^{2\pi} \int_{D_{\rho}(0)}(J_0+J \cos(\theta-\thetat))\, \left(\Rb+c(t)e^{i(q\thetat+k\cdot (x+\xt))}\right)\, \frac{\xt}{|\xt|} d^2\xt\, d\thetat  \\
&=e^{ikx}c(t) \int_{D_{\rho}(0)}\int_0^{2\pi}(J_0+J \cos(\theta-\thetat))\, e^{iq\thetat}\,d\thetat\, e^{ik\cdot \xt}\, \frac{\xt}{|\xt|} d^2\xt\\
&=\begin{cases}
\pi c(t)Je^{i(q\theta+k\cdot x)}\int_{D_\rho(0)}e^{ik\cdot \xt}\frac{\xt}{|\xt|}d^2\xt & \text{ for } q=\pm 1\\
2\pi c(t)J_0e^{ik\cdot x}\int_{D_\rho(0)}e^{ik\cdot \xt}\frac{\xt}{|\xt|}d^2\xt & \text{ for } q=0\\
0 & \text{ otherwise}
\end{cases}\label{vappendix}
\end{align}
 \end{linenomath}
Here the identity $\int_0^{2\pi}e^{iq\thetat}\cos(\theta-\thetat)d\thetat=\begin{cases}\pi e^{iq\theta} & q=\pm 1\\ 0 &\text{otherwise}\end{cases}$ was used.
To evaluate the integral in the above formula for $\v(t,\theta,x;R)$ in the case $k\neq 0$, rewrite the integral in polar coordinates $\xt=r\cos\phi\, \frac{k}{|k|}+r\sin\phi\, k^{\perp}$, where 
$k^{\perp}$ is a unit vector perpendicular to $k$. We then obtain
 \begin{linenomath}
\begin{align*}
  \int_{D_\rho(0)}e^{ik\cdot \xt}\frac{\xt}{|\xt|}d^2\xt&=\int_0^{2\pi}\int_0^\rho e^{i|k|r\cos\phi}r(\cos\phi\, \frac{k}{|k|}+\sin\phi\, k^{\perp})dr\,d\phi\\
 &=i\,\rho^2 f(|k|,\rho)\frac{k}{|k|},
\end{align*}
 \end{linenomath}
where   $f(|k|,\rho)=\frac{1}{\rho^2}\int_0^\rho\int_0^{2\pi}\sin(|k|r\cos\phi)r\cos\phi\, d\phi\, dr=\int_0^1\int_0^{2\pi}\sin(\rho|k|r\cos\phi)r\cos\phi\, d\phi\, dr$
 is a real-valued function. (The real part of the integral evaluate to zero, as can be seen by symmetry considerations and periodicity of sine and cosine.) Note two things: First, $f(|k|=0,\rho)=0$, so the functional form of $f$ also applies to the case $k=0$
in (\ref{vappendix}). Second, $f(|k|,\rho)$ actually only depends on the (dimensionless) product $k^*=\rho|k|$.
This gives
\begin{linenomath}
\begin{align*}
\nabla\cdot(\v(t,\theta,x;R)R)=\begin{cases}
-\pi J\Rb |k|\rho^2 f(|k|,\rho) e^{i(q\theta+k\cdot x)}c(t)  +\mathcal{O}(c(t)^2)  & \text{ for } q=\pm 1\\
-2\pi J_0\Rb |k|\rho^2  f(|k|,\rho) e^{ik\cdot x}c(t)  +\mathcal{O}(c(t)^2)  & \text{ for } q=0\\
0 & \text{ otherwise}\end{cases}
\end{align*}
\end{linenomath}

\bigskip
\noindent{\underline{\bf{Term }  $\frac{\partial}{\partial\theta}(\w(t,\theta,x;R)R)$}}

\medskip
One computes that 
 \begin{linenomath}
\begin{align*}
\w(t,x,\theta;R)&=K\int_{0}^{2\pi} \int_{D_{\rho}(0)} \sin(\thetat-\theta))\, \left(\Rb+c(t)e^{i(q\thetat+k\cdot (x+\xt))}\right)\, d^2\xt\, d\thetat  \\
&=K\, c(t)e^{i k\cdot x}\int_{D_{\rho}(0)} \int_0^{2\pi}\sin(\thetat-\theta))e^{iq\thetat}d\thetat\, e^{i k\cdot\xt}\, d^2\xt\\
&=\begin{cases}
i\pi\, q\, K c(t) e^{i(q\theta+k\cdot x)}\int_{D_\rho(0)}e^{ik\cdot \xt}d^2\xt & \text{ for } q=\pm 1\\
0 & \text{ otherwise}
\end{cases}
\end{align*}
 \end{linenomath}
Here the identity $\int_0^{2\pi}e^{iq\thetat}\sin(\thetat-\theta)d\thetat=\begin{cases}i\,q\pi e^{iq\theta} & q=\pm 1\\ 0 &\text{otherwise}\end{cases}$ was used.
To evaluate the integral in the formula for $\w(t,x,\theta;R)$, switch again to polar coordinates  $x=r\cos\phi\, \frac{k}{|k|}, y+r\sin\phi\, k^{\perp}$. This gives
\begin{linenomath}
\begin{align*}
  \int_{D_\rho(0)}e^{ik\cdot \xt}d^2\xt&=\int_0^{2\pi}\int_0^\rho e^{i|k|r\cos\phi}rdr\,d\phi=\rho^2\, g(|k|,\rho).
\end{align*}
\end{linenomath}
Here $g(|k|,\rho)=\frac{1}{\rho^2}\int_0^{2\pi}\int_0^\rho \cos({|k|r\cos\phi})rdr\,d\phi=\int_0^{2\pi}\int_0^1 \cos({\rho|k|r\cos\phi})rdr\,d\phi$ is a real-valued function.
Note that  $g(|k|,\rho)$ actually only depends on the (dimensionless) product $k^*=\rho|k|$.
This gives
 \begin{linenomath}
\begin{align*}
\frac{\partial}{\partial\theta}(\w(t,\theta,x;R)R)=\begin{cases}
-\pi K\Rb\rho^2 g(|k|,\rho) e^{i(q\theta+k\cdot x)}c(t)  +\mathcal{O}(c(t)^2)  & \text{ for } q=\pm 1\\
0 & \text{ otherwise}\end{cases}
\end{align*}
 \end{linenomath}

\bigskip
\noindent{\underline{\bf{Complete linearized equation}}
By inserting (\ref{linR}) into (\ref{diffeq}), using the above computations and neglecting terms of second order and higher in $c(t)$, we obtain the result of Proposition~\ref{linearization} in $d=2$ spatial dimensions:

\begin{proposition}
\item For $q=0$, the linearization is
\[
    \frac{d c}{dt}=(-|k|^2\D+2 J_0\Rb\rho\,  F_2(|\rho k|))\, c  
\]
\item For $|q|\neq 0, 1$, the  linearization is
\[
        \frac{d c}{dt}=(-|k|^2\D-q^2\Dc)\, c
\]

\item For $|q|=1$, the  linearization is
\[
        \frac{d c}{dt}=(-|k|^2 \D-\Dc+ J_0\Rb\rho\,  F_2(|\rho k|) + \Rb K\rho^2\, G_2(|\rho k|) )\, c.
\]
\end{proposition} 
Explicitly, the two functions $F_2(k^*)$ and $G_2(k^*)$ for real $k^*\geq 0$ are given by the following formulas, where we used Mathematica to 
get explicit expressions for the integrals:
 \begin{linenomath}
\begin{align*}
F_2(k^*)&=\pi k^*\, f(k^*,1)=\pi k^*\int_0^1\int_0^{2\pi}\sin(k^*r\cos\phi)r\cos\phi\, d\phi\, dr\\
&=  \pi^3 ({\mathcal{J}}_1(k^*)\, H_0(k^*)-\mathcal{J}_0(k^*)\, H_1(k^*) )
\end{align*}
 \end{linenomath}
and
 \begin{linenomath}
\begin{align*}
G_2(k^*)&=\pi g(k^*,1)=\pi \int_0^1\int_0^{2\pi}\cos(k^*r\cos\phi)r d\phi dr\\
&=  2\pi^2 \frac{1}{k^*}{\mathcal{J}}_1(k^*).
\end{align*}
 \end{linenomath}
Here $\mathcal{J}_0, \mathcal{J}_1$ denotes Bessel functions of the first and second kind, respectively, and 
$H_0,H_1$ denote Struve functions \cite{NISTsite}. Plots of $F_2$ and $G_2$ are shown in Fig~\ref{figFG}.

\section{Constraints on $J_0^*, J^*,K^*$}\label{constr_j0jk}
As pointed out in the main tetxt, the PDE model  (\ref{diffeq}) hinges on the assumption that patterns have sufficiently large characteristic scales that a description using cell density 
is appropriate. More specifically, we are only interested in values of $J_0, J$ and $K$ which lead to patterns with a characteristic length scale comparable to the cell interaction radius $\rho$  
or larger. In this Appendix, we identify a region in parameter space which satisfies this constraints. 

For this, consider the linearization discussed in section~\ref{linstab}. For a linearization pattern (mode) with spatial
 wave vector $k$, a characteristic length scale is the distance between peaks and valleys $\ell=\pi/|k|$. The condition that $\ell$ does not exceed the
interaction radius $\rho$ is thus equivalent to $k^*=\rho k<\pi$. Proposition~\ref{linearization}
shows that the growth rates of different modes are governed by the functions $\sigma_1(k^*)$ and $\sigma_3(k^*)$ given in $d$ dimensions by
 \begin{linenomath}
\begin{align*}
\sigma_1(k^*)&=-(k^*)^2+2 J_0^*\,  F_d(k^*)\\
\sigma_3(k^*)&=-(k^*)^2 -\Dc^*+ J^*  F_d(k^*) + K^* G_d(k^*) 
\end{align*}
 \end{linenomath}
where
\[
J^*=\frac{\bar{R}\rho^{d+1}}{\D}J, \quad J_0^*=\frac{\bar{R}\rho^{d+1}}{\D}J_0,\quad 
K^*=\frac{\bar{R}\rho^{d+2}}{\D}K
\] 
The dominant modes are those
that correspond to global maxima of those functions.

We now use the following condition to constrain the parameters $J_0,J$ and $K$:
We require 
\begin{equation}\label{cond}
\max_{k^*\geq\pi}(\sigma_i(k^*))<\max(0,\max_{k^*\geq 0}\sigma_i(k^*))\quad\text{for }i=1,3
\end{equation}
This somewhat cumbersome formula simply means that we require that the dominant modes $k^*_{\rm{dom}} $for both $\sigma_1$ and $\sigma_3$ to 
either be less than $\pi$ or have negative temporal growth rates. In other words, a pattern with length scale less than the interaction length $\rho$ cannot be the fastest growing pattern.

Using the explicit forms of $F_d(k^*)$ and $G_d(k^*)$ given in section~\ref{linstab}, one can derive with a combination of simple calculus arguments and numerical investigation the following results.
(In dimension 2, these are not sharp.)

\begin{proposition}
\begin{enumerate}
\item \label{range1} In $d=1$ spatial dimensions, the condition (\ref{cond}) is satified in the following parameter range:
\[
           -\pi<K^*<\infty, \quad -\infty<J_0^*,J^*<\infty.
\]
\item \label{range2}In $d=2$ spatial dimensions, the condition (\ref{cond}) is satifies in the following parameter range:
\[
           -1.4<K^*<\infty ,\quad -7.3<J^*_0<3.1,\quad -3.65<J^*<1.9
\]
\end{enumerate}
\end{proposition}
(For (\ref{range1}), note that $K^*>-\pi$ guarantees $\sigma_3(\pi)<0,$ which implies that the condition (\ref{cond}) is satisfied.
The result in  (\ref{range2}) was obtained numerically using Mathematica.)

\section{Solutions of the form $R(t,x,\theta)=R(t,\theta)$}\label{appendixtheta}
We consider the special case that the cell density $R(t,x,\theta)$ in (\ref{diffeq}) is constant with respect to the spatial position $x$:
\[
R(t,x,\theta)=R(t,\theta).
\]
Strictly speaking, this is 
of course only applicable in the case that the initial conditions are perfectly spatially homogeneous.
But the analysis is more generally relevant in the case that all spatially dependent modes have negative growth rates, see Propositions~\ref{propj0} and \ref{propjk}.
The differential equation for $R(t,\theta)$ takes the form
\begin{equation}\label{diffeqtheta}
\frac{\partial R}{\partial t}=\Dc \frac{\partial^2 R}{\partial\theta^2}-\,\frac{\partial }{\partial\theta}(\w(t,\theta;R)R).
\end{equation}
with 
\newcommand{\Dr}{|D_\rho|}
\begin{equation*}
\w(t,\theta;R)=K\,\Dr \int_{0}^{2\pi} \sin(\thetat-\theta)\, R(t,\thetat)\, d\thetat ,
\end{equation*}
where $\Dr$ is the $n-$dimensional area of the ball of radius $\rho$ (i.e. the length of the interaction neighborhood in one spatial dimension or its area in two dimensions.)

Consider the Fourier expansion of $R(t,\theta)$:
\begin{equation}\label{fourier}
R(t,\theta)=\sum_{j=-\infty}^{\infty}A_j(t)e^{ij\theta}.
\end{equation}
We consider real valued $R(t,\theta)$, so $A_{-j}(t)=\bar{A}_j(t)$. Note that $A_j(t=0)$ are the Fourier coefficients of the
initial condition $R(t=0,\theta)=R_0(\theta)$.

We have the following proposition:
\begin{proposition}
The Fourier coefficients $A_j(t)$ satisfy the infinite system of differential equations
\begin{equation}\label{systemaj}
\frac{d A_j}{dt}=j\Dr \pi K(A_1(t)A_{j-1}(t)-\bar{A}_1(t)A_{j+1}(t))-j^2\Dc A_j(t) \quad \quad {\text{ for }}j\in\mathbb{Z}
\end{equation}
In particular
\[
\frac{d}{dt}A_0(t)=0,
\]
so that $A_0(t)=const=\frac{1}{2\pi}\int_0^{2\pi}R_0(\theta)d\theta=\Rb$
\end{proposition}
\begin{proof}
Compute first $\w(t,\theta;R)$ with the Fourier expansion~\ref{fourier}:
 \begin{linenomath}
\begin{align*}
\w(t,\theta;R)& = \frac{K}{2i}\,\Dr \int_{0}^{2\pi}(e^{i(\thetat-\theta)})-e^{-i(\thetat-\theta))}) \, R(t,\thetat)\, d\thetat \\
& = \frac{K}{2i}\,\Dr \sum_{j=-\infty}^{\infty}A_j(t)\left(e^{-i\theta}\int_{0}^{2\pi} e^{i(j+1)\thetat}\,d\thetat  
-e^{i\theta}\int_{0}^{2\pi} e^{i(j-1)\thetat}d\thetat\right) \\
& = i\pi K\Dr\left(-A_{-1}(t)e^{-i\theta}+A_1(t)e^{i\theta}\right).
\end{align*}
\end{linenomath}
Using this expression in the differential equation (\ref{diffeqtheta}) yields
 \begin{linenomath}
\begin{align*}
\frac{\partial R}{\partial t}&=\sum_{j=-\infty}^{\infty} \left(K\Dr\pi\left( -(j-1)A_{-1}(t)A_j(t)e^{i(j-1)\theta}+(j+1)A_1(t)A_j(t)e^{i(j+1)\theta}\right)-j^2\Dc A_j(t) e^{ij\theta}\right)
\end{align*}
\end{linenomath}
Comparing Fourier coefficients of the left and right hand sides gives the system of equations for $A_j(t)$ in (\ref{systemaj})
\end{proof}

\begin{theorem}\label{thmconvincoherent}
For $K<0$ and $\Dc>0$, the incoherent state of the space-independent problem (\ref{diffeqtheta}) 
is globally stable, i.e. $R(t,\theta)$ converges to the constant $\Rb=\frac{1}{2\pi}\int_0^{2\pi}R(t=0,\theta)d\theta$ as $t\to\infty$.
\end{theorem}
\begin{proof}
Consider the function
\[
u(t)=\sum_{k=1}^{\infty}\frac{1}{k}|A_k|^2
\]
which converges absolutely by Parseval's Theorem. One computes that
\[
\frac{1}{k}\frac{d}{dt}|A_k|^2=K\Dr\pi\( A_1A_{k-1}\bar{A}_k-\bar{A}_1A_{k+1}\bar{A}_{k}+\bar{A}_1\bar{A}_{k-1}A_k-A_1\bar{A}_{k+1}A_k  \)-2\Dc k|A_k|^2.
\]
Using this, one finds that 
 \begin{linenomath}
\begin{align*}
\frac{du}{dt}&=\sum_{k=1}^{N}\frac{1}{k}\frac{d}{dt}|A_k|^2\\ 
&=K\Dr \pi\left( 2A_0|A_1|^2-\bar{A}_1A_{N+1}\bar{A}_N-A_1\bar{A}_{N+1}A_N\right)-2\Dc k|A_k|^2
\end{align*}
\end{linenomath}
due to a telescoping sum. Note that the remainder goes to zero
due to the fact that $\sum_{k=1}^{\infty}|A_k|^2$ converges. Therefore
\[
  \frac{du}{dt}=2K\Dr \pi A_0|A_1|^2-2\Dc \sum_{k=1}^{\infty}k|A_k|^2
\] 
If $K<0$, we can estimate
\[
\frac{du}{dt}\leq -2\Dc u(t).
\]
Provided $\Dc>0$, Gronwall's lemma implies $u(t)\to 0$ as $t\to\infty$. This means that $|A_k(t)|^2\to 0$
as $t\to\infty$ for all $k=1,2,\ldots$. Thus
\[
R(t,\theta)\to A_0=\frac{1}{2\pi}\int_0^{2\pi}R(t=0,\theta)d\theta.
\]
\end{proof}
We strongly expect the result to still hold for the case of zero clock diffusion ($\Dc=0$). In this case
\[
   \frac{du}{dt}=2K\Dr \pi A_0|A_1|^2
\]
This means that $u(t)=\sum_{k=1}^{\infty}\frac{1}{k}|A_k|^2$ is strictly decreasing, as $A_0>0$ since we assume nonnegative initial conditions. We conjecture that in fact $u(t)\to 0$ as $t\to\infty$. This would then
imply the statement of Theorem~ \ref{thmconvincoherent}.

\end{document}